\documentclass[12pt]{amsart}
\usepackage{a4wide,enumerate,xcolor}
\usepackage{amsmath,graphicx}
\allowdisplaybreaks
\usepackage{amssymb,amsthm}

\usepackage[colorlinks=true,linkcolor=blue,
anchorcolor=black,citecolor=cyan,filecolor=black,menucolor=black,runcolor=black,
urlcolor=blue]{hyperref}
\usepackage{comment}
\includecomment{versiona}

\let\pa\partial
\let\na\nabla
\let\eps\varepsilon
\newcommand{\N}{{\mathbb N}}
\newcommand{\R}{{\mathbb R}}
\newcommand{\Rge}{{\mathbb R}_\ge} 
\newcommand{\diver}{\operatorname{div}}
\newcommand{\T}{{\mathcal T}}
\newcommand{\E}{{\mathcal E}}
\newcommand{\m}{\mathfrak{m}}
\newcommand{\dist}{{\operatorname{d}}}
\newcommand{\dd}{\mathrm{d}}
\newcommand{\mm}{\mathfrak{\widetilde m}}
\DeclareMathOperator*\esssup{ess\,sup}

\newtheorem{theorem}{Theorem}
\newtheorem{lemma}[theorem]{Lemma}
\newtheorem{proposition}[theorem]{Proposition}
\newtheorem{remark}[theorem]{Remark}
\newtheorem{corollary}[theorem]{Corollary}
\newtheorem{definition}{Definition}

\newcommand{\Shn}{S}
\newcommand{\shn}{\mathfrak{h}}
\newcommand{\BB}{B}

\newcommand{\SShannon}{Shannon\ }

\newcommand{\Om}{\Omega}

\newcommand{\qbf}{p}
\newcommand{\ra}{\rangle}\newcommand{\la}{\langle}

\makeatletter
\@namedef{subjclassname@2020}{\textup{2020} Mathematics Subject Classification}
\makeatother

\begin{document}

\title[Finite-volume scheme for hyperbolic--parabolic systems]{
	Convergence of a finite-volume scheme and \\
    dissipative measure-valued--strong stability for \\
	a hyperbolic--parabolic cross-diffusion system}

\author[K. Hopf]{Katharina Hopf}
\address{Katharina Hopf, Weierstra{\ss} Institute for Applied Analysis and Stochastics (WIAS),
Mohrenstr. 39, 10117 Berlin, Germany}
\email{hopf@wias-berlin.de}

\author[A. J\"ungel]{Ansgar J\"ungel}
\address{Ansgar J\"ungel, Institute of Analysis and Scientific Computing, Technische Universit\"at Wien,
Wiedner Hauptstra\ss e 8--10, 1040 Wien, Austria}
\email{juengel@tuwien.ac.at}

\date{\today}

\thanks{The second author acknowledges partial support from
the Austrian Science Fund (FWF), grants P33010 and F65.
He thanks Antoine Zurek (Compi\`egne) and Flore Nabet (Palaiseau) for helpful discussions
and pointing out the papers \cite{Nab16,Oul18}.
This work has received funding from the European
Research Council (ERC) under the European Union's Horizon 2020 research and innovation programme, ERC Advanced Grant no.~101018153.}

\begin{abstract}
This article is concerned with the development of a theoretical framework of global measure-valued solutions for a class of hyperbolic--parabolic cross-diffusion systems, and its application to the convergence analysis of a fully discrete finite-volume scheme. 
After introducing an appropriate notion of dissipative measure-valued solutions to the PDE system, a numerical scheme is proposed which is shown to generate, in the continuum limit, a dissipative measure-valued solution. 
The ``parabolic density part'' of the  limiting measure-valued solution is atomic
and converges to its constant state for long times. 
Furthermore, it is proved that whenever the PDE system possesses a strong solution, the convergence of the approximation scheme holds in the strong sense.
The results are based on Young measure theory and a weak--strong stability estimate combining Shannon and Rao entropies.
The convergence of the numerical scheme is achieved by means of discrete entropy dissipation inequalities and an artificial diffusion, which vanishes in the continuum limit.
\end{abstract}



\maketitle

\section{Introduction}

The segregation of multi-species populations can be modeled at a macroscopic level by cross-diffusion equations. Segregation typically requires the associated diffusion matrix to have a nontrivial kernel. In this situation, solutions may have spatial discontinuities; see, e.g., \cite{BGHP85} for a two-species model. 
Diffusively regularized segregation models have been derived, for an arbitrary number of species, from interacting particle systems in a mean-field-type limit~\cite{CDJ19}. 
The class considered here has recently been found to possess a symmetric hyperbolic--parabolic structure~\cite{DHJ22}.
In this paper, we establish the global existence of dissipative measure-valued solutions as a limit of finite-volume approximations, 
the uniqueness of strong solutions among dissipative measure-valued solutions, and a result on the  long-time asymptotic behavior.

\subsection{Equations}

The segregation cross-diffusion equations for the vector $u=(u_1,\ldots,u_n)$ of
the population densities $u_i$ are systems of continuity equations,
\begin{equation}\label{1.eq.A}
  \pa_t u_i + \diver(u_i\mathsf{v}_i)=0, \quad \mathsf{v}_i=-\na p_i(u), 
	\quad\mbox{in }\Omega,\ t>0,\ i=1,\ldots,n,
\end{equation}
where $p_i(u) = \sum_{j=1}^n a_{ij}u_j$ and $\Omega\subset\R^d$ ($d\ge 1$) is a bounded Lipschitz domain, supplemented with the
no-flux boundary and initial conditions
\begin{equation}\label{1.bic}
  u_i\na p_i(u)\cdot\nu = 0\quad\mbox{on }\pa\Omega,\ t>0, \quad
	u_i(0)=u_i^{\rm in}\quad\mbox{in }\Omega,\ i=1,\ldots,n,
\end{equation}
where $\nu$ denotes the exterior unit normal vector to $\pa\Omega$. 
The variables $(u_i)$ represent, for instance, densities of animal populations \cite{BGHP85},
healthy and tumor cell densities \cite{LLP_2017}, or heights of thin fluid layers 
\cite{EMM12,LM_2023}.

The parameters $a_{ij}\ge 0$ are assumed to satisfy the following two conditions:
The matrix $A=(a_{ij})\in\R^{n\times n}$ is semistable, i.e., the real parts of
all its eigenvalues are nonnegative, and it satisfies the detailed-balance condition,
i.e., there exist $\pi_1,\ldots,\pi_n>0$ such that
\begin{equation}\label{1.dbc}
  \pi_i a_{ij} = \pi_j a_{ji}\quad\mbox{for all }i,j=1,\ldots,n,\ i\neq j.
\end{equation}
These equations can be recognized as the detailed-balance condition for the Markov chain associated to $A$, and the vector $(\pi_i)$ is an invariant measure.
Under condition~\eqref{1.dbc}, the change of variables $u_j\mapsto \pi_ju_j=:\widetilde u_j$ brings the equation in the form $\pa_t\widetilde u_i=\diver(\widetilde u_i\nabla (B\widetilde u)_i)$, where the matrix $\BB=(a_{ij}\pi_j^{-1})_{ij}$ is symmetric and positive semidefinite. Thus, from now on we consider, without loss of generality, the equations
\begin{equation}\label{1.eq}
	\pa_t u_i = \diver(u_i\na p_i(u)), \quad p_i(u) = \sum_{j=1}^n b_{ij}u_j
	\quad\mbox{in }\Omega,\ t>0,\ i=1,\ldots,n,
\end{equation}
where $\BB=(b_{ij})\in\R^{n\times n}$ is symmetric positive semidefinite and $b_{ij}\ge 0$ for all $i,j=1,\dots,n$.
We note that if $b_{ii}=0$ for some $i$, then $b_{ji}=b_{ij}=0$ for $j=1,\dots,n$ due to the positive semidefiniteness of $B$. Thus, in this case,
the dynamics of $u_i$ become trivial and the $i$th species can be removed from the system. We may therefore further assume that $b_{ii}>0$ for all $i=1,\dots,n$.
If $\operatorname{rank}\BB=n$ and $u_i>0$ for all $i=1,\dots,n$, equation~\eqref{1.eq} is parabolic in the sense of Petrovskii, 
which at the linear level is a minimal condition for the
generation of an analytic semigroup on $L^p(\Omega)$ \cite{Ama93}.
The existence of global weak solutions in the case $\operatorname{rank}\BB=n$ was
investigated in~\cite[Theorem~17]{JPZ22}. 
If $B$ has a nontrivial kernel, it is positive definite only on the subspace $(\operatorname{ker}B)^\perp$, 
and we lose the parabolic structure.
This is the situation we are primarily concerned with in this paper.

\subsection{State of the art}

Equations \eqref{1.eq.A} with nontrivial kernel of $A$ have been studied in the literature in special cases. 
The first work is \cite{BGHP85}, where the global existence of segregated solutions for two species in one space dimension with $a_{11}=a_{12}=1$ and $a_{21}=a_{22}=k>0$ was shown.
This result relies on a change to mass variables. 
The analysis was generalized in \cite{BHIM12} to several space dimensions if $k=1$. 
The idea is to introduce new variables $w_1=u_1+u_2$ and $w_2=u_1/(u_1+u_2)$.
It turns out that $w_1$ solves a porous-medium equation with quadratic nonlinearity
and $w_2$ solves a transport equation, demonstrating the hyperbolic--parabolic nature of
the system. The same idea was used in \cite{CFSS18} for a related system with general pressures $p_i$ and, employing different techniques, in \cite{GPS19} with $p_i(u)=(u_1+u_2)^\gamma$ for $\gamma>1$. Notice that the choice $k=1$ means that the corresponding velocity fields $\textsf{v}_i$ in~\eqref{1.eq.A} are independent of $i$, so that the motion of the two species is governed by a single velocity field.

The existence of an infinite family of minimizers of the entropy 
(or free energy) functional for different local and nonlocal variants was proved in \cite{BCPS20}, showing that both segregation and mixing of species is possible.
If the pressure is the variational derivative of a certain functional,
one may formulate \eqref{1.eq.A} for $n=2$ as a formal gradient flow. This
property has been exploited in \cite{BCPS20,DEF18} to prove the convergence of a minimizing scheme. 

The one-velocity two-species case was generalized to an arbitrary number of species 
in \cite{DrJu19}, proving the global existence
of classical and weak solutions by decomposing the system
into one decoupled porous-medium equation and $n-1$ transport equations. This approach was generalized in \cite{DHJ22} to the case of multiple velocity fields and with associated diffusion matrices of arbitrary rank $r\in\{1,\dots,n\}$ to show the local existence of classical solutions. Segregating solutions for one-velocity multi-species reactive systems were constructed in~\cite{Jac22}.

There exist related cross-diffusion models with rank-deficient diffusion matrices
in the literature, for instance the Maxwell--Stefan equations for fluid mixtures
\cite{Bot11}, where the diffusion matrix has a one-dimensional kernel. In contrast to the present problem, the kernel can
be removed by taking into account the volume-filling assumption $\sum_{i=1}^n u_i=1$, which allows one to reduce the system for the densities $u_1,\ldots,u_n$ 
to a parabolic one for the variables $u_1,\ldots,u_{n-1}$ via
$u_n=1-\sum_{i=1}^{n-1}u_i$ \cite{JuSt13}. 

The analysis and the convergence of approximation schemes to equations \eqref{1.eq.A} for general rank-deficient matrices $A$ 
is challenging, since the decomposition of the parabolic and hyperbolic parts is involved. Moreover, in view of the
results of \cite{BGHP85}, we cannot expect weak solutions in $H^1(\Omega)$,
and the hyperbolic part makes it difficult to obtain (entropy) solutions
in the distributional sense.
In the present paper, we choose to enlarge the solution space by 
considering dissipative measure-valued solutions, which allow us to encode information 
about the oscillation properties of the approximate solutions. 

DiPerna introduced the concept of (entropy) {\em measure-valued solutions} to
conservation laws \cite{DiP85}. In this framework, solutions are no longer integrable functions
but Young measures (parametrized probability measures), which are able to capture
the limiting behavior of sequences of oscillating functions. This concept is based
on an earlier work by Tartar \cite{Tar79}, who characterized weak limits of sequences of bounded functions. 
Due to the lack of uniqueness results, the framework of measure-valued solutions does
not allow one to identify the physically relevant solutions, and further
structural conditions on the solutions are necessary. 

One idea to resolve this issue is to require
an integrated form of the entropy or energy inequality, which leads to the 
concept of {\em dissipative solutions}. It has been introduced by P.-L.~Lions
\cite[Sec.~4.4]{Lio96} in the context of the incompressible Euler equations.
In~\cite{BDS_2011} it is shown that dissipative measure-valued solutions to the incompressible Euler equations enjoy the weak--strong uniqueness property, i.e., the dissipative measure-valued solution is atomic and coincides with
the strong or classical solution of the same initial-value problem if the
latter exists. 
This idea was further applied to models from polyconvex elastodynamics~\cite{DST12}, to the compressible Euler and Navier--Stokes equations \cite{FGSW16,GSW_2015},
to hyperbolic--parabolic systems in thermoviscoelasticity \cite{ChTz18}, and to various other, mainly fluid mechanical models.
	
In the present paper, we obtain dissipative measure-valued solutions to \eqref{1.eq}, \eqref{1.bic} by passing to the limit from discrete {\em finite-volume solutions}. 
We further show that they enjoy the weak--strong uniqueness property (in the sense of measure-valued--strong uniqueness),
which entails important consequences for the numerical approximation. 
Indeed, one may expect that reasonable structure-preserving approximation schemes generate a dissipative measure-valued solution. If such a measure-valued solution turns out to be atomic, i.e.\ taking the form of a Dirac measure at each point in space-time, 
Young measure theory implies that the underlying approximate solutions  converge in the strong sense. This idea has, for instance, been exploited in the proof of the convergence of finite-volume-type schemes for the compressible Navier--Stokes and Euler equations \cite{FeLu18,FLM20}.
For a further discussion on the use of measure-valued solutions in the numerical context, we refer to~\cite{FKMT_2016}.

The novelty of this paper is the analysis of equations \eqref{1.eq} with
general rank-deficient matrices $B$ by combining the
measure-valued framework, entropy methods, and finite-volume schemes.

\subsection{Key tools, definitions, and overview}

The analysis of \eqref{1.eq} is based on the observation that the system possesses
two Lyapunov functionals, respectively, the  \SShannon and Rao entropies
\begin{align}
  & H_\Shn(u) = \int_\Omega h_\Shn(u)\dd x, \quad
	h_\Shn(u) = \sum_{i=1}^n\big(u_i(\log u_i-1)+1\big), 
    \label{1.HB} \\
	& H_R(u) = \int_\Omega h_R(u)\dd x, \quad
	h_R(u) = \frac12\sum_{i,j=1}^nb_{ij} u_iu_j. \label{1.HR}
\end{align}
The Shannon (--Boltzmann) entropy is related to the thermodynamic entropy of the system,
while the Rao entropy measures the functional diversity of the species
\cite{Rao82}. 

The functionals have two important properties. First, a computation shows that, 
along smooth solutions to \eqref{1.eq}, \eqref{1.bic},
\begin{align}
  \frac{\dd H_\Shn}{\dd t}(u) 
	+ \sum_{i,j=1}^n\int_\Omega b_{ij}\na u_i\cdot\na u_j \dd x &= 0, 
	\label{1.eiB} \\
	\frac{\dd H_R}{\dd t}(u) + \sum_{i=1}^n\int_\Omega u_i|\na p_i(u)|^2\dd x &= 0.
	\label{1.eiR}
\end{align}
Since the matrix $\BB$ is positive semidefinite, the Shannon entropy dissipation
term (the integral term in \eqref{1.eiB}) is nonnegative and consequently,
$t\mapsto H_\Shn(u(t))$ is nonincreasing.
The expression $p_i(u)$ can be interpreted as the $i$th partial pressure and
$-\na p_i(u)$ as the $i$th partial velocity (by Darcy's law). Thus, we may interpret
the Rao entropy dissipation integral as the total kinetic energy of the system,
and $t\mapsto H_R(u(t))$ is also nonincreasing.

Second, the Shannon and Rao entropy densities $h_S$ and $h_R$ are convex, and their sum 
$h_S+h_R$ is strictly convex and has quadratic growth as $|u|\to\infty$, $u\in(0,\infty)^n$, as soon as $b_{ij}\ge0$ and $b_{ii}>0$ for all $i,j=1,\dots,n$. 
These properties allow us  to derive a weak--strong stability estimate based on the Bregman distance $h(u|v):=h(u)-h(v)-h'(v)\cdot(u{-}v)$ associated with $h=h_S+h_R$. 

Identities \eqref{1.eiB}--\eqref{1.eiR}
provide estimates for $u_i$ in $L^\infty(0,T;L^2(\Omega))$ 
and for $(\BB u)_i$ in $L^2(0,T;$ $H^1(\Omega))$, $T>0$.
If $\BB$ is rank-deficient, 
these bounds do not ensure gradient estimates for the whole vector $u$.
Notice that the weak convergence for $u_m$ and $\na p_i(u_m)=\na(\BB u_m)_i$ in $L^2(\Omega\times(0,T))$, which may be expected for suitable approximating sequences $u_m$,  does not allow us to identify the weak limit of $u_{m,i}\na(\BB u_m)_i$. This
issue is overcome by a suitable concept of dissipative measure-valued solutions. Let us mention that the estimates coming from~\eqref{1.eiR} lead to a control of $u_{m,i}\na(\BB u_m)_i$ in $L^2(0,T;L^{4/3}(\Om))$, thus ruling out potential concentrations in this term.

Before giving the definition of the measure-valued solutions, we introduce some notation. We rewrite equation \eqref{1.eq} as
$$
   \pa_t u_i = \diver(u_i\na(\BB u)_i), \quad i=1,\ldots,n,
$$
and set $L:=\operatorname{ker}\BB\subsetneq\mathbb{R}^n$. Then 
$L^\perp=\operatorname{ran}B$. Let $P_{L^\perp}$ be the projection
onto $L^\perp$ and set $\widehat{s}:=P_{L^\perp}s$ for $s\in\R^n$.
Any vector-valued function $u$ is written as $u=(u_1,\ldots,u_n)$.
We define $\Rge=[0,\infty)$ and let $\mathcal{P}(W)$ be the space of
probability measures on 
$$
  W := \Rge^n\times(L^\perp)^d.
$$
The space $L^\infty_w(\Omega\times[0,\infty);\mathcal{P}(W))$ is the set of 
weakly$^*$ measurable, essentially bounded functions of $\Omega\times[0,\infty)$
taking values in $\mathcal{P}(W)$. 
We henceforth use the notation
	$$
	\langle\nu,f(s,p)\rangle :=\int_{W}f(s,p)\,\dd\nu(s,p)\quad\text{for }\nu\in \mathcal{P}(W),f\in C_0(W),
	$$
	where $C_0$ denotes the space of
	continuous functions vanishing at infinity. Whenever the right-hand side is well defined, this notation will also be used for more general continuous functions $f$.
Finally, we let $\Omega_T:=\Omega\times(0,T)$ for $T>0$.

\begin{definition}[Dissipative measure-valued solution]\label{def.dmvs}
Suppose that $u^{\rm in}\in L^2(\Omega;\Rge^n)$.
We call a parametrized measure 
$$
  \mu\in L^\infty_w(\Omega\times[0,\infty);\mathcal{P}(W))
$$
with barycenters $u:=\langle\mu,s\rangle$, $y:=\langle\mu,p\rangle$
a {\em dissipative measure-valued solution} to \eqref{1.eq}, \eqref{1.bic} 
if the following is satisfied for all $T>0$:
\begin{itemize}
\item {\rm Regularity:} 
It holds that
\begin{align*}
  &\qquad {u}\in L^\infty(0,\infty;L^2(\Omega;{\R^n})),\ \partial_tu\in L^2(0,\infty;W^{1,4}(\Omega;{\R^n})^*),
  \ y\in L^2(\Omega_T;(L^\perp)^d), \ y=\na\widehat{u}.
\end{align*}
Moreover, $\mu$ acts trivially on the $\widehat{s}$-component,
\begin{equation}\label{2.sepa}
	\langle\mu,f(s,p)\rangle = \langle\mu,f(\widehat{u}+P_Ls,p)\rangle
\end{equation}
for all $f\in C_0(\Rge^n\times (L^\perp)^d)$.
\item {\rm \SShannon and Rao entropy inequalities:} It holds for a.e.\ $t>0$ that
\begin{align}
	H_S^{\rm mv}(u(t))
	+ \int_0^t\int_\Omega\la\mu_{x,\tau},|\BB^{1/2}p|^2\ra \dd x\dd\tau
	&\le H_\Shn(u^{\rm in}), \label{1.eiBm} \\
	H_R(u(t)) + \sum_{i=1}^n\int_0^t\int_\Omega\big\langle\mu_{x,\tau},
	s_i|(Bp)_i|^2\big\rangle\,\dd x\dd \tau &\le H_R(u^{\rm in}), \label{1.eiRm}
\end{align}
where $H_\Shn$ and $H_R$ are defined in \eqref{1.HB}--\eqref{1.HR}
and $H_S^{\rm mv}(u(t)):=\int_\Om\la \mu_{x,t},h_S(s)\ra\dd x$.
\item {\rm Evolution equation:} It holds for all $i=1,\ldots,n$ and
$\phi\in C^1_c(\overline\Omega\times[0,T))$ that
\begin{equation}\label{1.evol}
  \int_0^T\int_\Omega u_i\pa_t\phi\dd x\dd t + \int_\Omega u_i^{\rm in}\phi(0)\dd x
  = \int_0^T\int_\Omega\langle\mu_{x,t},s_i(\BB p)_i\rangle\cdot 
  \na\phi\dd x\dd t.
\end{equation}
\end{itemize}
\end{definition}

It is easy to see that, under the hypotheses of Definition~\ref{def.dmvs}, 
the term $\langle\mu,s_i(\BB p)_i\rangle\in L^1(\Omega_T)$ is well defined for all $T>0$
 (cf.~Section~\ref{ssec:tempreg}). Moreover, $u_i=\la\mu,s_i\ra\ge0$.
 Property~\eqref{2.sepa} can be extended to a larger class of continuous functions $f$.
In particular, it holds for all $f\in C(W)$ with $f\ge0$.
If ${\rm rank}\,B=n $, property~\eqref{2.sepa} implies that $u$ fulfills \eqref{1.eq}, \eqref{1.bic} in the usual weak sense,  since then $P_L=0$.
We further observe in the following remark the consistency of Definition~\ref{def.dmvs} with the standard weak solution concept.
\begin{remark}[Consistency of the definition]\label{rem.dmvs}\rm
The definition of dissipative measure-valued solutions is consistent with the
definition of weak solutions. Indeed, any weak solution $u$ to 
\eqref{1.eq}, \eqref{1.bic} satisfying the regularity statements
of Definition \ref{def.dmvs} and the  \SShannon and
Rao entropy inequalities gives rise to a dissipative measure-valued solution $\mu$
via $\mu_{x,t}=\delta_{u(x,t)}\otimes\delta_{\na\widehat{u}(x,t)}$. 
On the other hand, if a dissipative measure-valued solution $\mu$ is trivial 
in the sense that $\mu_{x,t}=\delta_{v(x,t)}\otimes\delta_{z(x,t)}$ for certain
functions $v$ and $z$, then $v=\langle\mu,s\rangle=u$ and $z=\langle\mu,p\rangle
=y=\na\widehat{u}$. We infer that
$$
\langle\mu,s_i(\BB p)_i\rangle 
= u_i(\BB\na\widehat{u})_i.
$$
In this case, equation \eqref{1.evol} reduces to the standard
weak formulation of \eqref{1.eq} for the density $u$ and the entropy inequalities \eqref{1.eiBm} and \eqref{1.eiRm} take the usual form of entropy inequalities for weak solutions.
More generally, the conclusion 
$\langle\mu,s_i(\BB p)_i\rangle 
= u_i(\BB\na\widehat{u})_i$ already holds if, for instance, $\mu$ is only atomic in the density component, i.e. $\mu_{x,t}=\delta_{v(x,t)}\otimes\nu_{x,t}$, where $\nu$ denotes the parametrized measure generated by $(\nabla\widehat{u}_m)_m$ with $(u_m)_m$ denoting the approximate sequence.
\end{remark}

Our main results can be sketched as follows; we refer to Section \ref{sec.main} for the precise statements. 

\begin{itemize}
\item Existence of finite-volume approximations:
There exists a sequence of approximate solutions $(u_m)$,
where $m\in\N$ indicates the fineness of the mesh, 
to an implicit Euler finite-volume scheme.
The numerical scheme preserves the structure of the equations, namely nonnegativity,
conservation of mass, and entropy dissipation; see Theorem \ref{thm.ex}.

\item Existence of global dissipative measure-valued solutions:
Any Young measure $\mu$ generated by $(u_m)$ is a dissipative measure-valued solution
 to \eqref{1.eq}, \eqref{1.bic} in the sense of Definition \ref{def.dmvs};
see Theorem \ref{thm.conv}. For this result, we need to include some artificial diffusion in the scheme, which vanishes in the limit $m\to\infty$.

\item Weak--strong uniqueness: If $v$ is a positive classical solution 
to \eqref{1.eq}, \eqref{1.bic}
with initial datum $v(0)=u^{\rm in}$ and $\mu$ is a dissipative measure-valued solution to
\eqref{1.eq}, \eqref{1.bic},
then $\mu_{x,t}=\delta_{v(x,t)}\otimes\delta_{\na\widehat{v}(x,t)}$ for a.e.\
$(x,t)\in\Omega_T$; see Theorem \ref{thm.wsu}.

\item Long-time behavior: The density 
$\widehat{u}(t):=\langle\mu_{\boldsymbol{\cdot},t},\widehat{s}\rangle$ 
converges strongly in the $L^2(\Omega)$ norm as $t\to\infty$
to a function $\widehat{u}^*\in L^2(\Omega;[0,\infty)^n)$
satisfying $\int_\Omega\widehat{u}^*\dd x=\int_\Omega u^{\rm in}\dd x$ and
$\na(\BB\widehat{u}^*)=0$ in $\Omega$; see Theorem \ref{thm.time}.
\end{itemize}

If equations \eqref{1.eq}, \eqref{1.bic} admit a classical solution, 
the weak--strong uniqueness property implies that the sequence
of finite-volume solutions converges, in the strong $L^1$-sense, to this classical solution on the
lifespan of the latter; see Corollary~\ref{cor:strong}.

We stress the fact that, while the existence of solutions is proved via a finite-volume scheme, the weak--strong uniqueness and long-time behavior results are independent of the numerical scheme. In this regard, the discrete numerical approximation serves as a tool for the existence analysis, even though the mere existence of dissipative measure-valued solutions may more readily be obtained via approximation by a regularized continuous system.

The paper is organized as follows. We introduce the numerical scheme
and the precise statements of the theorems in Section \ref{sec.numer}.
The four theorems are proved in Sections \ref{sec.ex}--\ref{sec.time},
and we conclude in Appendix \ref{sec.app} with some auxiliary lemmas.

\section{Numerical scheme and main results}\label{sec.numer}

First, we introduce the notation necessary to formulate our numerical method. Then we state the numerical scheme and the main results.

\subsection{Spatial domain and mesh}

Let $d\ge 2$ and let $\Omega\subset\R^d$ be a bounded polygonal domain 
(or polyhedral if $d\ge 3$).
We associate to this domain an admissible mesh, given by
(i) a family $\T$ of open polygonal (or polyhedral) control volumes, which are
also called cells,  
(ii) a family $\E$ of edges (or faces if $d\ge 3$),
and (iii) a family of points $(x_K)_{K\in\T}$ associated to the control volumes
and satisfying \cite[Definition 9.1]{EGH00}. This definition implies that
the straight line $\overline{x_Kx_L}$ between two centers of neighboring cells
is orthogonal to the edge (or face) $\sigma=K|L$ between two cells.
For instance, Vorono\"{\i} meshes satisfy this condition \cite[Example 9.2]{EGH00}.
The size of the mesh is given by $\Delta x = \max_{K\in\T}\operatorname{diam}(K)$.
The family of edges $\E$ is assumed to consist of interior edges
$\E_{\rm int}$ satisfying $\sigma\subset \Omega $ and boundary edges $\sigma\in\E_{\rm ext}$
satisfying $\sigma\subset\pa\Omega$. For a given $K\in\T$, $\E_K$ denotes the 
set of edges of $K$, splitting into $\E_K= \E_{{\rm int},K}\cup\E_{{\rm ext},K}$.
For any $\sigma\in\E$, there exists at least one cell $K\in\T$ such that
$\sigma\in\E_K$. 

We need a regularity assumption on the families of meshes we admit.
For given $\sigma\in\E$, 
we define the distance
$$
  \dist_\sigma = \begin{cases}
	\dist(x_K,x_L) &\quad\mbox{if }\sigma=K|L\in\E_{{\rm int},K}, \\
	\dist(x_K,\sigma) &\quad\mbox{if }\sigma\in\E_{{\rm ext},K},
	\end{cases}
$$
where d is the Euclidean distance in $\R^d$, and the transmissibility coefficient
\begin{equation}\label{2.trans}
  \tau_\sigma = \frac{\mm(\sigma)}{\dist_\sigma},
\end{equation}
where $\mm(\sigma)$ denotes the $(d{-}1)$-dimensional Hausdorff measure of $\sigma$. 
We suppose the following mesh regularity condition for any admissible family of meshes $\{\T\}$: There exists a fixed
$\zeta>0$, independent of $\T$, such that for all $K\in\T$ and $\sigma\in\E_K$,
\begin{equation}\label{2.regul}
  \dist(x_K,\sigma)\ge \zeta \dist_\sigma.
\end{equation}
This condition means that the family of meshes $\{\T\}$ is (locally) quasi-uniform.
We also use the geometric property
\begin{equation}\label{2.sum}
  \sum_{\sigma\in\E_{{\rm int},K}}\mm(\sigma)\dist(x_K,\sigma)
	\le d\m(K)\quad\mbox{for any }K\in\T,
\end{equation}
where $\m$ denotes the $d$-dimensional Lebesgue measure.
Inequalities \eqref{2.regul} and \eqref{2.sum} are needed, for instance, 
to derive a uniform bound 
for the discrete time derivative of the approximate solution; 
see Lemma \ref{lem.time}. 

\subsection{Function spaces} \label{ssec:function-spaces}

Let $T>0$, $N\in\N$ and introduce the time step size $\Delta t=T/N$
and the time steps $t_k=k\Delta t$ for $k=0,\ldots,N$. 
We denote by
$\mathcal{D}$ the space-time discretization of 
$\Omega_T=\Omega\times(0,T)$ determined by the mesh $\T$ and by
the values $(\Delta t,N)$.

The space of piecewise constant functions is defined by
$$
  V_\T = \bigg\{v: \Omega\to\R:\exists (v_K)_{K\in\T}\subset\R,\
	v(x)=\sum_{K\in\T}v_K\mathbf{1}_K(x)\bigg\},
$$
where $\mathbf{1}_K$ is the characteristic function on $K$. To define a norm
on this space, we define for $v\in V_\T$, $K\in\T$, and $\sigma\in\E_K$,
$$
  v_{K,\sigma} = \begin{cases}
	v_L &\quad\mbox{if }\sigma=K|L\in\E_{{\rm int},K}, \\
	v_K &\quad\mbox{if }\sigma\in\E_{{\rm ext},K},
	\end{cases} \quad
	\textrm{D}_{K,\sigma} v := v_{K,\sigma}-v_K, \quad 
	\textrm{D}_\sigma v := |\mathrm{D}_{K,\sigma} v|.
$$
Let $1\le q<\infty$ and $v\in V_\T$.
The discrete $W^{1,q}(\Omega)$ norm on $V_\T$ is given by 
\begin{align*}
  & \|v\|_{1,q,\T} = \big(\|v\|_{0,q,\T}^q+|v|_{1,q,\T}^q\big)^{1/q}, 
	\quad\mbox{where} \\
	& \|v\|_{0,q,\T}^q = \sum_{K\in\T}\m(K)|v_K|^q, \quad
  |v|_{1,q,\T}^q = \sum_{\sigma\in\E}\mm(\sigma)\dist_\sigma
	\bigg|\frac{\text{D}_{\sigma} v}{\dist_\sigma}\bigg|^q,
\end{align*}
where $v\in V_\T$. If $v=(v_1,\ldots,v_n)\in V_\T^n$ is a vector-valued function, 
we write for notational convenience
$$
  \|v\|_{1,q,\T} = \sum_{i=1}^n\|v_i\|_{1,q,\T}.
$$
We associate to the discrete $W^{1,q}$ norm a dual norm with respect 
to the $L^2$ inner product:
$$
  \|v\|_{-1,q,\T} = \sup\bigg\{\int_\Omega vw \dd x: w\in V_\T,\
	\|w\|_{1,q,\T}=1\bigg\}.
$$
Then the following property holds:
$$
  \bigg|\int_\Omega vw \dd x\bigg| \le \|v\|_{-1,q,\T}\|w\|_{1,q,\T}
	\quad\mbox{for all }v,w\in V_\T,\   1<p<\infty.
$$

Finally, we introduce  the space $V_{\T,\Delta t}$ of piecewise constant
functions with values in $V_\T$,
\begin{equation*}
  V_{\T,\Delta t} = \bigg\{v:\Omega\times[0,T]\to\R:
	\exists(v^k)_{k=1,\ldots,N}\subset V_\T,\
	v(x,t) = \sum_{k=1}^{N} v^k(x)\mathbf{1}_{(t_{k-1},t_k]}(t)\bigg\},
\end{equation*}
equipped with the discrete $L^2(0,T;H^1(\Omega))$ norm
$$
  \left(\sum_{k=1}^{N}\Delta t \|v^k\|_{1,2,\T}^{2}\right)^{1/2} \quad 
	\mbox{for all }v\in V_{\T,\Delta t}.
$$ 

\subsection{Discrete gradient}

The discrete gradient is defined on a dual mesh. For this, we 
define the cell $T_{K,\sigma}$ of the dual mesh for $K\in\T$ and $\sigma\in\E_K$:
\begin{itemize}
\item Diamond: Let $\sigma=K|L\in \E_{{\rm int},K}$. Then $T_{K,\sigma}$ is that cell
whose vertices are given by $x_K$, $x_L$, and the end points of the edge $\sigma$.
\item Triangle: Let $\sigma\in\E_{{\rm ext},K}$. Then $T_{K,\sigma}$ is that cell
whose vertices are given by $x_K$ and the end points of the edge $\sigma$.
\end{itemize}
The union of all diamonds and triangles $T_{K,\sigma}$ equals the domain $\Omega$
(up to a set of measure zero).
The property that the straight line $\overline{x_Kx_L}$ between two neighboring centers
of cells is orthogonal to the edge $\sigma=K|L$ implies that
\begin{equation*}
  \mm(\sigma)\,\dist(x_K,x_L)=d\,\m(T_{K,\sigma})\quad\mbox{for all }
	\sigma=K|L\in\E_{\rm int}.
\end{equation*}
The approximate gradient of $v\in V_{\T,\Delta t}$ is then defined by
$$
  \na^{\mathcal{D}} v(x,t) = \frac{\mm(\sigma)}{\m(T_{K,\sigma})}
	\mathrm{D}_{K,\sigma}(v^k)\nu_{K,\sigma}
	\quad\mbox{for }x\in T_{K,\sigma},\ t\in{(t_{k-1},t_k]},
$$
where $\nu_{K,\sigma}$ is the unit vector that is normal to $\sigma$ and
points outwards of $K$. 

\subsection{Numerical scheme}

The initial functions are approximated by $u^0\in V_\T^n$ defined via
\begin{equation}\label{2.init}
  u_{i,K}^0 = \frac{1}{\m(K)}\int_K u_i^{\rm in}(x)dx \quad\mbox{for all } K\in\T,\ 
	i=1,\ldots,n.
\end{equation}
Let $u^{k-1}=(u^{k-1}_1,\dots,u^{k-1}_n)\in  V_\T^n$ be given. Then the values $u_{i,K}^k$ for all $K\in\T$ and $i=1,\ldots,n$ are determined
by the implicit Euler finite-volume scheme
\begin{align}\label{2.fvm}
  & \m(K)\frac{u_{i,K}^k-u_{i,K}^{k-1}}{\Delta t} 
	+ \sum_{\sigma\in\E_K}\mathcal{F}_{i,K,\sigma}^k = 0, \\
  \label{2.flux}
  & \mathcal{F}_{i,K,\sigma}^k = -\tau_\sigma u_{i,\sigma}^k
	\textrm{D}_{K,\sigma} p_i(u^k) 
    - \eta^\alpha\tau_\sigma\textrm{D}_{K,\sigma}u_i^k,
\end{align}
where $\eta=\max\{\Delta x,\Delta t\}$, $0<\alpha<2$, and
$\tau_\sigma$ is given by \eqref{2.trans}.
The mobility $u_{i,\sigma}^k$ is defined for $\sigma\in\E$ by the upwind scheme
\begin{equation}\label{2.upwind}
  u_{i,\sigma}^k =  \begin{cases}
  u_{i,K,\sigma}^k &\quad\mbox{if }\textrm{D}_{K,\sigma}p_i(u^k)\ge 0, \\
  u_{i,K}^k &\quad\mbox{if }\textrm{D}_{K,\sigma}p_i(u^k)< 0.
  \end{cases}
\end{equation}
The upwind approximation allows us to derive the discrete \SShannon entropy inequality; see Remark \ref{rem.upwind}. We may also use a logarithmic mean function; see Remark \ref{rem.log}.

We have added some artificial diffusion in the numerical flux 
$\mathcal{F}_{i,K,\sigma}^k$, which vanishes in the limit $\eta\to 0$.
The term is needed to show the convergence of the scheme. In particular, 
it provides an $\eta$-dependent bound for the full gradient, 
compensating the incomplete gradient estimate.
Note that the artificial diffusion is {\em not} needed
to prove the existence of discrete solutions, and we may set $\eta=0$ in this case. Artificial diffusion/viscosity is used in numerical approximations of the Euler equations to stabilize the scheme; see, e.g., \cite[(3.8)]{FLM20}.

The numerical fluxes $\mathcal{F}_{i,K,\sigma}^k$ are consistent approximations of the exact fluxes through the edges, since $\mathcal{F}_{i,K,\sigma}+\mathcal{F}_{i,L,\sigma}=0$ for all edges $\sigma=K|L$ and $\mathcal{F}_{i,K,\sigma}=0$ for all $\E_{{\rm ext},K}$.
The following discrete integration-by-parts formula holds for $v=(v_K)\in V_\T$:
\begin{equation}\label{2.ibp}
  \sum_{K\in\T}\sum_{\sigma\in\E_K}\mathcal{F}_{i,K,\sigma}v_K
	= -\sum_{\sigma\in\E_{\rm int}}\mathcal{F}_{i,K,\sigma} \mathrm{D}_{K,\sigma}v.
\end{equation}
Notice that the terms $\mathcal{F}_{i,K,\sigma}\mathrm{D}_{K,\sigma}v$ on the right-hand side only depend on $\sigma$, but not on the specific control volume $K$ satisfying $\sigma\in \mathcal{E}_K$. Hence, to evaluate the sum on the right, we may pick for each $\sigma$ any $K$ with $\sigma\in \mathcal{E}_K$ as long as we keep $K$ fixed.

\begin{remark}[Discrete gradient-flow property for upwind scheme]
\label{rem.upwind}\rm 
The upwind approximation implies a discrete gradient-flow property. Indeed, we first observe that the concavity of the logarithm gives
$$
  b(\log a-\log b)\le a-b\le a(\log a-\log b)\quad\mbox{for all }a,b>0.
$$
Combined with definition \eqref{2.upwind} of $u_{i,\sigma}^k$, this leads for $u_{i,K}^k>0$ and $u_{i,L}^k>0$ to
\begin{equation}\label{2.gf}
  u_{i,\sigma}^k(p_i(u_L^k)-p_i(u_K^k))(\log u_{i,L}^k-\log u_{i,K}^k)
  \ge (p_i(u_L^k)-p_i(u_K^k))(u_{i,L}^k - u_{i,K}^k)
\end{equation}
and therefore, by discrete integration by parts \eqref{2.ibp},
\begin{align}\label{2.remup}
  \sum_{i=1}^n&\sum_{K\in\T}\sum_{\sigma\in\E_K}
  \mathcal{F}_{i,K,\sigma}^k\log u_{i,K}^k 
  = -\sum_{i=1}^n\sum_{\sigma\in\E_{\rm int}}\tau_\sigma u_{i,\sigma}^k
  \textrm{D}_{K,\sigma}p_i(u^k)\textrm{D}_{K,\sigma}\log u_i^k \\
  &{}- \eta^{\alpha}\sum_{i=1}^n\sum_{\sigma\in\E_{\rm int}}\tau_\sigma
  \textrm{D}_{K,\sigma}u_i^k\textrm{D}_{K,\sigma}\log u_i^k
  \le -\sum_{i=1}^n\sum_{\sigma\in\E_{\rm int}}\tau_\sigma b_{ij}
  \textrm{D}_{K,\sigma}u_j^k\textrm{D}_{K,\sigma}u_i^k, \nonumber 
\end{align}
where we used the monotonicity of the logarithm implying that
$\textrm{D}_{K,\sigma}u_i^k\textrm{D}_{K,\sigma}\log u_i^k\ge 0$.
The right-hand side of \eqref{2.remup} is nonpositive due to the positive semidefiniteness of $B=(b_{ij})$. We deduce from this inequality the discrete entropy inequality \eqref{2.eiB}.
\end{remark}

\begin{remark}[Discrete gradient flow property for logarithmic mean]
\label{rem.log}\rm
We may alternatively define $u_{i,\sigma}^k$ via the logarithmic mean
\begin{equation}\label{2.logmean}
  u_{i,\sigma}^k = \begin{cases}\displaystyle
	\frac{u_{i,L}^k-u_{i,K}^k}{\log u_{i,L}^k-\log u_{i,K}^k}
	&\quad\mbox{if }u_{i,K}^k\neq u_{i,L}^k\mbox{ and }u_{i,K}^k>0,\ u_{i,L}^k>0, \\
	u_{i,K}^k &\quad\mbox{if }u_{i,K}^k=u_{i,L}^k>0, \\
	0 &\quad\mbox{else}.
	\end{cases}
\end{equation}
We remark that the artificial diffusion in the numerical flux \eqref{2.flux} allows us to show that $u_{i,K}^k$ is positive for all $K\in\T$ (see Section \ref{sec.pos}) such that $u_{i,\sigma}^k$ (for $\sigma=K|L$) is always defined by one of the first two cases.
Definition \eqref{2.logmean} also leads to a discrete gradient-flow property. Indeed, observing that $u_{i,\sigma}^k\mathrm{D}_{K,\sigma}\log u_i^k
= \mathrm{D}_{K,\sigma} u_i^k$ and 
multiplying \eqref{2.flux} by $\log u_{i,K}^k$ and
summing over all $i=1,\ldots,n$, $K\in\T$, and $\sigma\in\E_K$, we 
 see that \eqref{2.remup} holds too.
Notice that \eqref{2.gf} becomes an equality in this case. 
\end{remark}

Finally, we observe that the mobility satisfies in both cases the following properties:
\begin{equation}\label{2.meanprop}
  u_{i,\sigma}^k\le\max\{u_{i,K}^k,u_{i,L}^k\}, \quad
  |u_{i,\sigma}^k-u_{i,K}^k| \le |u_{i,K}^k-u_{i,L}^k|
\quad\mbox{for }\sigma=K|L.
\end{equation}

\subsection{Main results}\label{sec.main}

We impose the following hypotheses.

\begin{itemize} 
\item[(H1)] Data: $\Omega\subset\R^d$ is a bounded polygonal 
(or polyhedral if $d\ge 3$) domain, $T>0$, and $u^{\rm in}\in L^2(\Omega;\Rge^n)$ such that $\|u^{\rm in}\|_{L^1(\Omega)}>0$.
We set $\Omega_T=\Omega\times(0,T)$.
\item[(H2)] Discretization: $\mathcal{D}$ is an admissible discretization of 
$\Omega_T$ satisfying \eqref{2.regul}.
\item[(H3)] Diffusion coefficients:
$\BB=(b_{ij})\in\Rge^{n\times n}$ is symmetric positive semidefinite with
$\operatorname{rank}\BB\in\{1,\dots, n\}$ and $b_{ii}>0$ for $i=1,\dots,n$.
\end{itemize}
Note that since $\BB$ is positive semidefinite, 
its square root $\BB^{1/2}$ exists and 
$z^T\BB z=|\BB^{1/2}z|^2$ for $z\in\R^n$. Moreover, with $\lambda>0$ being the smallest positive eigenvalue of $\BB^{1/2}$, we have $|\BB^{1/2}z|\ge\lambda|\widehat z|$ (cf.\ Lemma~\ref{lem.APL}).

\begin{theorem}\label{thm.ex}
Let Hypotheses (H1)--(H3) hold, $k\in\N$, $\eta\ge 0$,
and let $u^{k-1}\in V_\T^n$ be given. 
Then there exists a solution $u^k=(u_1^k,\ldots,u_n^k)
\in V_\T^n$ to scheme \eqref{2.init}--\eqref{2.flux} satisfying
 $u_{i,K}^k > 0$ for $i=1,\ldots,n$, $K\in\T$.
 Inductively, let $u^j\in V_\T^n, j=1,\dots,k$, be the solution to scheme \eqref{2.init}--\eqref{2.flux} with $u^{k-1}$ replaced by $u^{j-1}$. Then $\{u^j\}$ obey the discrete
entropy inequalities
\begin{align}
  H_\Shn(u^k) + \sum_{j=1}^k\Delta t	|\BB^{1/2}u^j|_{1,2,\T}^2 
	+ 4\eta^\alpha\sum_{j=1}^k\Delta t\sum_{i=1}^n |(u_i^j)^{1/2}|_{1,2,\T}^2
	&\le H_\Shn(u^0), \label{2.eiB} \\
  H_R(u^k) 
	+ \sum_{j=1}^k\Delta t\sum_{i=1}^n\sum_{\sigma\in\E}
	\tau_\sigma u_{i,\sigma}^j|\mathrm{D}_{\sigma}(\BB u^j)_i|^2 
	&\le H_R(u^0). \label{2.eiR}
\end{align}
Moreover, $ H_R(u^k) \le  H_R(u^{k-1})$. 
\end{theorem}

The existence of finite-volume solutions to \eqref{2.init}--\eqref{2.flux}
was shown in \cite{JuZu20} by using the Rao entropy only,
but the proof needs matrices $\BB$ with full rank.
We can avoid this condition since we exploit the estimates coming from 
the  \SShannon entropy. Theorem \ref{thm.ex} is proved by adding a discrete
version of the regularizing term $\eps(-\Delta w_i+w_i)$, where
$w_i= \log u_i$ are the entropic variables~\cite{FL_1971,Jue16,KS_1988}, 
and a topological degree argument, similar as in \cite{JuZu20}. 
Uniform estimates from the  \SShannon entropy inequality \eqref{2.eiB} allow us to perform the de-regularizing limit $\eps\to 0$.
Observe that the theorem is valid for $\eta=0$, i.e., no artificial diffusion is needed here.

 Theorem \ref{thm.ex} and the subsequent results also hold for 
domains $\Omega\subset\R^d$ with curved (Lipschitz) boundary. Indeed, one may triangulate $\Omega$ in such a way that the control volumes have a curved boundary \cite{Nab16}, or one may cover $\Omega$ by additional cells and estimate the integral error; we refer to Remark \ref{rem.curved} for details.

For the convergence result, 
we introduce a family $(\mathcal{D}_m)_{m\in\N}$ of admissible
space-time discretizations of $\Omega_T$ indexed by the size 
$\eta_m=\max\{\Delta x_m,\Delta t_m\}$ of the mesh, where 
$\Delta x_m=\max_{K\in\T_m}\operatorname{diam}(K)$ and $\Delta t_m$ is the
time step size of the mesh $\mathcal{D}_m$, satisfying $\eta_m\to 0$
as $m\to\infty$. We denote by $\T_m$ the corresponding meshes of $\Omega$
and set $\na^m:=\na^{\mathcal{D}_m}$. 

\begin{theorem}[Convergence of the scheme]\label{thm.conv}
Let Hypotheses (H1)--(H3) hold, and let $(\mathcal{D}_m)$ be a family
of admissible meshes satisfying \eqref{2.regul} uniformly in $m\in\N$.
Let $(u_m)$ be a sequence of finite-volume solutions to 
\eqref{2.init}--\eqref{2.flux} with $\eta=\eta_m>0$, 
constructed in Theorem~\ref{thm.ex}.
Then, up to a subsequence, $(u_m,\nabla^m\widehat u_m)$ generates a Young measure $\mu$ which is a dissipative measure-valued solution to \eqref{1.eq}, \eqref{1.bic} in the sense of Definition \ref{def.dmvs}.
 Moreover, the function
$t\mapsto H_R(u(t))$ is nonincreasing.
\end{theorem}

The strategy of the proof of Theorem \ref{thm.conv} is as follows.
The estimates from the discrete entropy inequalities and a uniform bound
for the discrete time derivative of $u_m$ allow us to apply the compactness result 
of \cite{GaLa12} to conclude the strong convergence of (a subsequence of)
$\widehat{u}_m$ in $L^2(\Omega_T)$ as $m\to\infty$. Moreover, $(u_m)$ 
and $\na^m(\BB\widehat{u}_m)$ are weakly converging in $L^2(\Omega_T)$. 
Clearly, these convergences are too weak to conclude the convergence 
of the nonlinear flux \eqref{2.flux}. 
However, the sequence $(u_m,\na^m\widehat{u}_m)$ 
generates a parametrized measure $\mu$ \cite[Chap.~6]{Ped97}
such that $\langle\mu,s_i(\BB p)_i\rangle$
is the distributional limit of $ u_{m,i,\sigma}\na^m(\BB\widehat{u}_m)_i$.
Moreover, because of the strong convergence of $(\widehat{u}_m)$, 
we can separate this part, leading to~\eqref{2.sepa}.

\begin{remark}[Full-rank approximation]\rm
Let $\Omega\subset\mathbb{R}^d$ be a bounded Lipschitz domain. An alternative to the finite-volume approach is to consider a suitable full-rank symmetric positive definite regularization $(\BB_\rho)\in \mathbb{R}^{n\times n}$ of $B$ with $\lim_{\rho\to 0}B_\rho=B$,
and to approximate~\eqref{1.eq} by
\begin{equation}\label{2.visc}
	\pa_t u_i = \diver(u_i\na(\BB_\rho u)_i), 
	\quad i=1,\ldots,n.
\end{equation}
After an appropriate additional regularization, it is possible to apply the entropy method of \cite[Sec.~4.4]{Jue16} (using the Rao entropy structure) and to establish the existence of a nonnegative weak solution to~\eqref{2.visc}, \eqref{1.bic} that satisfies both the Rao and Shannon entropy inequalities with $B$ replaced by $B_\rho$. The dissipative measure-valued solution to~\eqref{1.eq}, \eqref{1.bic} is then obtained in the limit $\rho\to 0$. 
\end{remark}

The statement of Theorem \ref{thm.conv} is rather weak, since the Young measure may not be unique. However, we can prove a weak--strong uniqueness result. According to Remark \ref{rem.curved}, we can assume in the following that $\Omega$ is a general bounded domain with Lipschitz boundary. 

\begin{theorem}[Weak--strong uniqueness]\label{thm.wsu}
Let $\Omega\subset\mathbb{R}^d$ be a bounded Lipschitz domain.
Let $v\in C^1(\overline\Omega\times[0,T];\Rge^n)$
be a positive solution to \eqref{1.eq}, \eqref{1.bic} (in the weak sense) with initial datum $v(0)=u^{\rm in}>0$,
and let $\mu$ be a dissipative measure-valued solution to \eqref{1.eq}, \eqref{1.bic}.
Then
$$
  \mu_{x,t} = \delta_{v(x,t)}\otimes\delta_{\na\widehat{v}(x,t)}\quad\mbox{for a.e. }
	(x,t)\in\Omega\times(0,T).
$$
\end{theorem}

The assertion is deduced from a stability estimate based on the 
Bregman distance $h(u|v):=h(u)-h(v)-h'(v)\cdot(u{-}v)$
associated with the convex function $h:=h_S+h_R$, which has to be adapted to the measure-valued framework.
Loosely speaking, we consider the sum $H_S(u|v)+H_{R}(u|v)$, where
$$
  H_k(u|v) = \sum_{i=1}^n\int_\Omega\big(h_k(u)-h_k(v)-h_k'(v)\cdot(u-v)\big)\dd x,
	\quad k=\Shn,R,
$$
and compute its time derivative along solutions to \eqref{1.eq}.
Certain error terms arising in this computation need to be estimated above by $C\int_\Omega h(u|v)\dd x$.  For this step and in the absence of $L^\infty(\Omega)$-bounds on the densities $u_i$, we take advantage of the better coercivity properties at infinity of the Rao entropy. 

As a consequence of Theorem \ref{thm.wsu}, the finite-volume solution converges \textit{strongly} to the classical solution if the latter exists. 

\begin{corollary}\label{cor:strong}
Let $u\in C^1(\overline\Omega\times[0,T];\Rge^n)$
be a positive solution to \eqref{1.eq}, \eqref{1.bic}.
Let $(u_m)$ be a sequence of finite-volume solutions to \eqref{2.init}, \eqref{2.flux} with $\eta=\eta_m>0$. Then, as $m\to\infty$,
\begin{align*}
	(u_m,\na^m\widehat{u}_m)\to (u,\nabla \widehat{u})\quad
		\text{strongly in $L^p(\Omega_T)$} 
\end{align*}
 for all $p\in [1,2)$ and all $T>0$.
\end{corollary}
Indeed, the weak--strong uniqueness implies that the Young measure generated by $(u_m,$ $\na^m\widehat{u}_m)$ coincides at each point $(x,t)$ with the Dirac measure concentrated at the smooth solution. Since $|(u_m,\na^m\widehat{u}_m)|^p\subset L^1(\Om_T)$ is equi-integrable for every $p\in[1,2)$, the assertion in Corollary~\ref{cor:strong} thus follows from classical Young measure theory (cf.~e.g.~\cite[Theorem~6.12]{Ped97}).

It is shown in~\cite[Theorem 2.6]{DHJ22} for $\Omega=\mathbb{T}^d$ (with periodic boundary conditions) that problem \eqref{1.eq}, \eqref{1.bic} possesses a positive classical solution on a short time interval if the initial data are positive and smooth. The main results in the present paper should equally be valid in the periodic setting.
 
If $B$ has a non-trivial kernel, steady states to \eqref{1.eq}, \eqref{1.bic} are not necessarily constant in space
and for any fixed mass vector $\mathsf{m}\in(0,\infty)^n$, there exist infinitely many steady states. Given such $\mathsf{m}$, we define the space of steady states as
$$
  \mathfrak{S}_\mathsf{m} = \bigg\{v\in L^2(\Omega;\Rge^n): \int_\Omega v\,\dd x=\mathsf{m}\mbox{ and }
	\na(\BB v)=0 \mbox{ in }\Omega\bigg\}.
$$
\begin{theorem}[Long-time behavior]\label{thm.time}
Let $\mu$ be a dissipative measure-valued solution to \eqref{1.eq}, \eqref{1.bic}.
Let $u=\langle\mu,s\rangle$ and set $\mathsf{m}:=\int_\Omega u^{\rm in}dx$. Then 
$\mathfrak{S}_\mathsf{m}\subset L^\infty(\Omega;\Rge^n)$ and there
exists $u^*\in \mathfrak{S}_\mathsf{m}$ such that, as $t\to\infty$,
$$
  \widehat{u}(t)\to \widehat{u}^* \quad\mbox{strongly in }L^2(\Omega;\Rge^n),
$$
where $\widehat{u}^*=P_{L^\perp}u^*$. We recall that $P_{L^\perp}$ is
the projection onto $L^\perp=\operatorname{ran}\BB$.
\end{theorem}

For the proof of Theorem~\ref{thm.time}, we argue as follows. The fact that $\int_0^\infty\|\na(\BB^{1/2}\widehat{u})\|^2_{L^2(\Omega)}\dd t$ is finite implies the existence of a sequence $t_k\to\infty$ such that
$k\mapsto(\BB^{1/2}u)(t_k)$ converges strongly in $L^2(\Omega)$ to
$\BB^{1/2}u^*$ as $k\to\infty$, where $u^*\in \mathfrak{S}_\mathsf{m}$. The monotonicity of $t\mapsto H_R(u(t)|u^*)$ then shows that $\BB^{1/2}\widehat{u}(t)$ converges and consequently, $\widehat{u}(t)$ converges to $\widehat{u}^*$ for all
sequences $t\to\infty$. Such reasoning is classical in degenerate cases, where entropy--entropy dissipation estimates are not available;
see for instance~\cite{CCLR_2016, Hopf_2022_SingularitiesFP}.

\section{Discrete problem}
\label{sec.ex}

In this section, we prove Theorem \ref{thm.ex}.
The existence proof uses a discrete analog of the entropy
method for cross-diffusion systems \cite{Jue16}.
We first introduce a regularized numerical scheme involving an approximation 
parameter $\eps>0$, prove the existence of a solution to this scheme and
suitable estimates coming from the  \SShannon entropy inequality, and apply 
a topological degree argument. The uniform estimates allow us to perform
the limit $\eps\to 0$.

\subsection{Definition and continuity of the fixed-point operator}

Let $u^{k-1}\in V_\T^n$ be given and let $R>0$, $\delta>0$. We set
$$
  Z_R = \big\{w=(w_1,\ldots,w_n)\in V_\T^n: \|w\|_{1,2,\T}<R
	\mbox{ for }i=1,\ldots,n\big\}
$$
and define the mapping $F:Z_R\to \R^{\theta n}$ by 
$F(w)=w^\eps$, where
$\theta=\#\T$ and $w^\eps=(w_1^\eps,\ldots,w^\eps_n)$ is the solution
to the linear regularized problem
\begin{equation}\label{3.lin}
  \eps\bigg(-\sum_{\sigma\in\E_K}\tau_\sigma\mathrm{D}_{K,\sigma} w_i^\eps
	+ \m(K)w_{i,K}^\eps\bigg) = -\bigg(\frac{\m(K)}{\Delta t}(u_{i,K} - u_{i,K}^{k-1})
	+ \sum_{\sigma\in\E_K}\mathcal{F}_{i,K,\sigma}\bigg),
\end{equation}
where $u_{i,K} := \exp(w_{i,K})$ and $\mathcal{F}_{i,K,\sigma}$ is defined
as in \eqref{2.flux} with $u_{i,K}^k$ replaced by $u_{i,K}$.

To show that $F$ is well defined, 
we write \eqref{3.lin} as
\begin{align}\label{3.M}
	\begin{aligned}
  & Mw^\eps = v, \quad\mbox{where } v = (v_{i,K})_{i=1,\ldots,n,\,K\in\T}, \\
	& v_{i,K} = \frac{\m(K)}{\Delta t}(u_{i,K} - u_{i,K}^{k-1})
	+ \sum_{\sigma\in\E_K}\mathcal{F}_{i,K,\sigma}, 
	\end{aligned}
\end{align}
and $M=\operatorname{diag}(M',\ldots,M')\in\R^{\theta n\times\theta n}$ 
is a block diagonal matrix with $M'\in\R^{\theta\times\theta}$, which has the entries
$$
  M'_{K,K} = -\eps\m(K) - \eps\sum_{\sigma\in\E_K}\tau_\sigma, \quad
	M'_{K,L} = \begin{cases}
	\eps\tau_\sigma &\quad\mbox{if }K\cap L\neq\emptyset,\ \sigma=K|L, \\
	0 &\quad\mbox{if }K\cap L=\emptyset.
	\end{cases}
$$
Therefore, the system $Mw^\eps=v$ can be decomposed into the independent subsystems
$M'w_i^\eps=v_i$ for $i=1,\ldots,n$. Since $M'$ is strictly diagonally dominant,
these subsystems possess a unique solution $w_i^\eps$. Then 
$w^\eps=(w_1^\eps,\ldots,w_n^\eps)$ is the unique solution to \eqref{3.M}.
Thus, the mapping $F$ is well defined. 

Next, we prove that $F$ is continuous. We multiply \eqref{3.lin} for some
fixed $i\in\{1,\ldots,n\}$ by $w_{i,K}^\eps$ and sum over all $i=1,\ldots,n$
and $K\in\T$:
\begin{align}\label{3.aux}
	\begin{aligned}
  -\eps&\sum_{i=1}^n\sum_{K\in\T}\sum_{\sigma\in\E_K}\tau_\sigma
	(\mathrm{D}_{K,\sigma}w_i^\eps)w_{i,K}^\eps 
	+ \eps\sum_{i=1}^n\sum_{K\in\T}\m(K)(w_{i,K}^\eps)^2 \\
	&= -\sum_{i=1}^n\sum_{K\in\T}\frac{\m(K)}{\Delta t}(u_{i,K}-u_{i,K}^{k-1})w_{i,K}^\eps
	- \sum_{i=1}^n\sum_{K\in\T}\sum_{\sigma\in\E_K}\mathcal{F}_{i,K,\sigma}w_{i,K}^\eps.
\end{aligned}
\end{align}
Using discrete integration by parts analogous to~\eqref{2.ibp}, we can rewrite the left-hand side as
\begin{align*}
  -\eps&\sum_{i=1}^n\sum_{K\in\T}\sum_{\sigma\in\E_K}\tau_\sigma
	(\mathrm{D}_{K,\sigma}w_i^\eps)w_{i,K}^\eps 
	+ \eps\sum_{i=1}^n\sum_{K\in\T}\m(K)(w_{i,K}^\eps)^2 \\
	&= \eps\sum_{i=1}^n\sum_{\sigma\in\E_{\rm int}}\tau_\sigma
	(\mathrm{D}_{K,\sigma}w_i^\eps)^2
	+ \eps\sum_{i=1}^n\sum_{K\in\T}\m(K)(w_{i,K}^\eps)^2 
	= \eps\sum_{i=1}^n\|w_i^\eps\|_{1,2,\T}^2.
\end{align*}

We turn to the terms on the right-hand side of \eqref{3.aux}.
By definition, we have $\|w_i\|_{1,2,\T}<R$ and consequently
$\|w_i\|_{0,\infty,\T}\le C(R,\T)$ and $\|u_i\|_{1,2,\T}\le C(R,\T)$
(since the problem is finite-dimensional).
This shows that
\begin{align*}
  -\sum_{i=1}^n\sum_{K\in\T}\frac{\m(K)}{\Delta t}(u_{i,K}-u_{i,K}^{k-1})
	w_{i,K}^\eps
	&\le \frac{1}{\Delta t}\sum_{i=1}^n\|u_{i}-u_{i}^{k-1}\|_{0,2,\T}
	\|w_{i}^\eps\|_{0,2,\T} \\
	&\le C(R,\T,\Delta t)\sum_{i=1}^n\|w_i^\eps\|_{1,2,\T}.
\end{align*}
Finally, using definition \eqref{2.flux} of the flux and discrete integration
by parts,
\begin{align*}
  -\sum_{i=1}^n&\sum_{K\in\T}\sum_{\sigma\in\E_K}\mathcal{F}_{i,K,\sigma}w_{i,K}^\eps
	= \sum_{i=1}^n\sum_{K\in\T}\sum_{\sigma\in\E_K}\tau_\sigma\bigg(\sum_{j=1}^n 
	b_{ij}u_{i,\sigma}(\mathrm{D}_{K,\sigma}u_{j}) 
	+ \eta^\alpha\mathrm{D}_{K,\sigma}u_i\bigg)w_{i,K}^\eps \\
	&= -\sum_{i=1}^n\sum_{\sigma\in\E_{\rm int}}
	\tau_\sigma\bigg(\sum_{j=1}^n b_{ij} u_{i,\sigma}
  (\mathrm{D}_{K,\sigma}u_{j})(\mathrm{D}_{K,\sigma}w_{i}^\eps)
	+ \eta^\alpha(\mathrm{D}_{K,\sigma}u_i)(\mathrm{D}_{K,\sigma}w_{i}^\eps)\bigg) \\
	&\le \max_{\sigma\in\E}\|u_{i,\sigma}\|_{0,\infty,\T}
	\sum_{i,j=1}^n b_{ij}|u_j|_{1,2,\T}|w_i^\eps|_{1,2,\T}
	+ \eta^\alpha\sum_{i=1}^n |u_i|_{1,2,\T}|w_i^\eps|_{1,2,\T} \\
	&\le C(R,\T)\|w_i^\eps\|_{1,2,\T}.
\end{align*}
For the last inequality, we used the fact that $u_{i,\sigma}$ depends on
$u_{i,K}$ and $u_{i,L}$ for $\sigma=K|L$, and their discrete $L^\infty(\Omega)$ norms can be
bounded by the discrete $L^\infty(\Omega)$ norm of $w_i$, which in turn can be estimated by $C(\T)\|w_i\|_{0,\infty,\T}\le C(R,\T)$. 

Inserting these estimates into \eqref{3.aux} and dividing by $\|w_i^\eps\|_{1,2,\T}$ if 
$\|w_i^\eps\|_{1,2,\T}>0$,
it follows that $\eps\|w_i^\eps\|_{1,2,\T}\le C(R,\T,\Delta t)$. This bound allows us to
verify the continuity of $F$. Indeed, let $w^\ell\to w$ as $\ell\to\infty$ and
set $w^{\eps,\ell}=F(w^\ell)$. Then $(w^{\eps,\ell})_{\ell\in\N}$ is uniformly
bounded in the discrete $H^1(\Omega)$ norm. Therefore, there exists a subsequence,
which is not relabeled, such that $w^{\eps,\ell}\to w^\eps$ as $\ell\to\infty$.
Passing to the limit $\ell\to\infty$ in scheme \eqref{3.lin}, we see that
$w^\eps$ is a solution to the scheme and $w^\eps=F(w)$. 
Since the solution to the
linear scheme \eqref{3.lin} is unique, the entire sequence $(w^{\eps,\ell})_{\ell\in\N}$
converges to $w^\eps$, which shows the continuity of $F$. 

\subsection{Existence of a fixed point}\label{sec.fp}

We will now show that the map $F$ admits a fixed point by using a topological degree
argument. We prove that $\operatorname{deg}(I-F,Z_R,0)=1$, where deg is the 
Brouwer topological degree \cite[Chap.~1]{Dei85}.
Since deg is invariant by homotopy, it is sufficient to verify that any solution
$(w^\eps,\rho)\in\overline{Z}_R\times[0,1]$ to the fixed-point equation
$w^\eps = \rho F(w^\eps)$ satisfies $(w^\eps,\rho)\not\in \pa Z_R\times[0,1]$
for sufficiently large values of $R>0$. Let $(w^\eps,\rho)$ be a fixed point.
The case $\rho=0$ being clear, we assume that $\rho\neq 0$. Then $w_i^\eps$
solves
\begin{equation}\label{3.fp}
  \eps\bigg(-\sum_{\sigma\in\E_K}\tau_\sigma\mathrm{D}_{K,\sigma} w_i^\eps
	+ \m(K)w_{i,K}^\eps\bigg) 
	= -\rho\bigg(\frac{\m(K)}{\Delta t}(u_{i,K}^\eps - u_{i,K}^{k-1})
	+ \sum_{\sigma\in\E_K}\mathcal{F}_{i,K,\sigma}^\eps\bigg)
\end{equation}
for $i=1,\ldots,n$ and $K\in\T$, where $u_{i,K}^\eps=\exp(w_{i,K}^\eps)$ and
$\mathcal{F}_{i,K,\sigma}^\eps$ is defined as in \eqref{2.flux} with
$u_{i,K}^k$ replaced by $u_{i,K}^\eps$. 
The following inequality is the key argument.

\begin{lemma}[Discrete  \SShannon entropy inequality]\label{lem.dei}
Let $w^\eps$ be a solution to \eqref{3.fp} and $u_i^\eps:=\exp(w_i^\eps)$. Then
\begin{align}\label{3.eiB}
	\begin{aligned}
  \rho H_\Shn(u^\eps)&+ \eps\Delta t\sum_{i=1}^n\|w_i^\eps\|_{1,2,\T}^2
	+ \rho\Delta t\sum_{i,j=1}^n\sum_{\sigma\in\E_{\rm int}}\tau_\sigma b_{ij}
	\mathrm{D}_{K,\sigma}u^\eps_i\mathrm{D}_{K,\sigma}u_j^\eps \\
	&{}+ 4\rho\eta^\alpha\Delta t\sum_{i=1}^n|(u_i^\eps)^{1/2}|_{1,2,\T}^2
	\le \rho H_\Shn(u^{k-1}).
\end{aligned}
\end{align}
\end{lemma}

\begin{proof}
We multiply \eqref{3.fp} by $\Delta t w_{i,K}^\eps$, sum over $i=1,\ldots,n$
and $K\in\T$, and use discrete integration by parts (cf.~\eqref{2.ibp}). 
Then
$\eps\Delta t\sum_{i=1}^n\|w_i^\eps\|_{1,2,\T}^2 = I_1+I_2+I_3$, where 
\begin{align*}
	& I_1 = -\rho\sum_{i=1}^n\sum_{K\in\T}\m(K)
	(u_{i,K}^\eps-u_{i,K}^{k-1})w_{i,K}^\eps, \\
	& I_2 = -\rho\Delta t\sum_{i=1}^n\sum_{\sigma\in\E_{\rm int}}\tau_\sigma
	u_{i,\sigma}^\eps\mathrm{D}_{K,\sigma} p_i(u^\eps)
	\mathrm{D}_{K,\sigma}w_{i,K}^\eps, \\
	& I_3 = -\rho\eta^\alpha \Delta t\sum_{i=1}^n\sum_{\sigma\in\E_{\rm int}}\tau_\sigma
	\mathrm{D}_{K,\sigma}u_i^\eps \mathrm{D}_{K,\sigma}w_{i,K}^\eps.
\end{align*}
The definition
$u_{i,K}^\eps=\exp(w_{i,K}^\eps)$ and the convexity of the \SShannon entropy imply that
$$
  I_1 = -\rho\sum_{i=1}^n\sum_{K\in\T}\m(K) (u_{i,K}^\eps-u_{i,K}^{k-1})
	\log u_{i,K}^\eps	\le -\rho\big(H_\Shn(u^\eps)-H_\Shn(u^{k-1})\big).
$$
For $I_2$, we rely on inequality \eqref{2.gf}:
\begin{align*}
  I_2 &= -\rho\Delta t\sum_{i=1}^n\sum_{\sigma=K|L\in\E_{\rm int}} 
	\tau_\sigma u_{i,\sigma}^\eps(p_i(u_L^\eps)-p_i(u_K^\eps))(\log u_{i,L}^\eps-\log u_{i,K}^\eps) \\
	&\le -\rho\Delta t\sum_{i,j=1}^n\sum_{\sigma=K|L\in\E_{\rm int}} 
	\tau_\sigma  b_{ij}(u_{j,L}^\eps-u_{j,K}^\eps)(u_{i,L}^\eps-u_{i,K}^\eps) \\
	&= -\rho\Delta t\sum_{i,j=1}^n\sum_{\sigma=K|L\in\E_{\rm int}} 
	\tau_\sigma  b_{ij}\mathrm{D}_{K,\sigma}u_i^\eps\mathrm{D}_{K,\sigma}u_j^\eps.
\end{align*}
Finally, using the elementary inequality 
$(a-b)(\log a-\log b)\ge 4(\sqrt{a}-\sqrt{b})^2$,
\begin{align*}
  I_3 &= -\rho\eta^\alpha\Delta t\sum_{i=1}^n
  \sum_{\sigma=K|L\in\E_{\rm int}}\tau_\sigma 
  (u_{i,L}^\eps-u_{i,K}^\eps)(\log u_{i,L}^\eps-\log u_{i,K}^\eps) \\
  &\le -4\rho\eta^\alpha\Delta t\sum_{i=1}^n
  \sum_{\sigma=K|L\in\E_{\rm int}}\tau_\sigma 
  \big((u_{i,L}^\eps)^{1/2}-(u_{i,K}^\eps)^{1/2}\big)^2
  = -4\rho\eta^\alpha\Delta t\sum_{i=1}^n |(u_i^\eps)^{1/2}|_{1,2,\T}^2.
\end{align*}
Combining these estimates finishes the proof of Lemma~\ref{lem.dei}.
\end{proof}

We now complete the topological degree argument. Lemma \ref{lem.dei} implies that
$$
  \eps\Delta t\sum_{i=1}^n\|w_i^\eps\|_{1,2,\T}^2
	\le \rho H_\Shn(u^{k-1}) \le H_\Shn(u^{k-1}).
$$
With the choice $R:=(\eps\Delta t)^{-1/2}H_\Shn(u^{k-1})^{1/2}+1$ we find that
$w^\eps\not\in\pa Z_R$ and $\operatorname{deg}(I-F,Z_R,0)=1$. We conclude that $F$ possesses a fixed point.

\subsection{Limit \texorpdfstring{$\eps\to 0$}{}} 

By Lemma \ref{lem.dei}, there exists $C>0$,
independent of $\eps$, such that
$$
  C\sum_{i=1}^n\sum_{K\in\T}\m(K)(u_{i,K}^\eps-1) \le H_\Shn(u^\eps) \le H_\Shn(u^{k-1}).
$$
This gives a uniform discrete $L^1(\Omega)$ bound for $u_{i}^\eps$. 
There exists a subsequence
(not relabeled) such that $u_{i,K}^\eps\to u_{i,K}$ as $\eps\to 0$ for all
$i=1,\ldots,n$ and $K\in\T$. Moreover, the discrete $H^1(\Omega)$ bound for
$\sqrt{\eps}w_i^\eps$ implies that $\eps w_{i,K}^\eps\to 0$ for $i=1,\ldots,n$
and $K\in\T$. Then the limit $\eps\to 0$ in \eqref{3.fp} yields the existence
of a solution $u^k:=(u_{i,K})_{i=1,\ldots,n,\,K\in\T}$ to \eqref{2.fvm}. 
Observing that
\begin{align*}
  \sum_{{i,j=1}}^n\sum_{\sigma\in\E}\tau_\sigma b_{ij}
	\mathrm{D}_{K,\sigma}u_i^\eps\mathrm{D}_{K,\sigma}u_j^\eps
	&= \sum_{\sigma\in\E}\tau_\sigma(\mathrm{D}_{K,\sigma}u^\eps)^T
	\BB(\mathrm{D}_{K,\sigma}u^\eps) \\
	&\ge \sum_{\sigma\in\E}\tau_\sigma|\BB^{1/2}\mathrm{D}_{K,\sigma}u^\eps|^2
	= |\BB^{1/2}u^\eps|_{1,2,\T}^2,
\end{align*}
the same limit in the regularized entropy inequality \eqref{3.eiB} 
directly leads to the discrete entropy inequality \eqref{2.eiB}.

\subsection{Discrete Rao entropy inequality}

To verify \eqref{2.eiR}, we multiply \eqref{2.fvm} by $\Delta t  p_i(u^k_K)$,
sum over $i=1,\ldots,n$ and $K\in\T$, and use discrete integration by parts:
\begin{align*}
  \sum_{i=1}^n&\sum_{K\in\T}\m(K)(u_{i,K}^k-u_{i,K}^{k-1})p_i(u^k_K) \\
	&= \Delta t\sum_{i=1}^n\sum_{K\in\T}\sum_{\sigma\in\E_K}
	\tau_\sigma \big(u_{i,\sigma}^k\mathrm{D}_{K,\sigma}p_i(u^k)
	+ \eta^\alpha\mathrm{D}_{K,\sigma}u_i^k\big)p_i(u^k_K) \\
	&= -\Delta t\sum_{i=1}^n\sum_{\sigma\in\E_{\rm int}}\tau_\sigma  u_{i,\sigma}^k
	(\mathrm{D}_\sigma p_i(u^k))^2
	- \eta^\alpha\Delta t\sum_{i,j=1}^n\sum_{\sigma\in\E_{\rm int}}\tau_\sigma b_{ij}
	\mathrm{D}_{K,\sigma}u_i^k\mathrm{D}_{K,\sigma}u_j^k \\
	&= -\Delta t\sum_{i=1}^n\sum_{\sigma\in\E_{\rm int}}\tau_\sigma  u_{i,\sigma}^k
	(\mathrm{D}_\sigma p_i(u^k))^2
	- \eta^\alpha\Delta t\sum_{i=1}^n|(\BB^{1/2}u^k)_i|_{1,2,\T}^2.
\end{align*}
By the definition of $p_i(u^k)$ and the symmetry and positive semidefiniteness of
$\BB$, the left-hand side becomes
\begin{align*}
  \sum_{i=1}^n&\sum_{K\in\T}\m(K)(u_{i,K}^k-u_{i,K}^{k-1})p_i(u^k_K) 
	= \sum_{i,j=1}^n\sum_{K\in\T}\m(K) b_{ij}(u_{i,K}^k-u_{i,K}^{k-1})u_{j,K}^k	\\
	&= \frac12\sum_{i,j=1}^n\sum_{K\in\T}\m(K) b_{ij} 
    \big(u_{i,K}^ku_{j,K}^k - u_{i,K}^{k-1}u_{j,K}^{k-1}
	+ (u_{i,K}^k-u_{i,K}^{k-1})(u_{j,K}^k-u_{j,K}^{k-1})\big) \\
	&\ge H_R(u^k)-H_R(u^{k-1}).
\end{align*}
We infer the monotonicity of $k\mapsto H_R(u^k)$. 
After summation over $k=1,\ldots,j$ and a renaming of the indices $k$ and $j$, this shows~\eqref{2.eiR} and thus completes the proof of Theorem~\ref{thm.ex}.

\subsection{Positivity}\label{sec.pos}

Thanks to the artificial diffusion, the discrete solution $u_{i,K}^k$ is positive for $i=1,\ldots,n$ and $K\in\T$. Indeed, let $i\in\{1,\ldots,n\}$ be fixed and assume that there exists $K\in\T$ such that $u_{i,K}^k=0$. We infer from $I_3$ in Section \ref{sec.fp} that
$$
  \eta^\alpha(u_{i,L}^\eps-u_{i,K}^\eps)(\log u_{i,L}^\eps-\log u_{i,K}^\eps) \le C(\Delta t,u^0),
$$
where $L\in K$ is a neighboring cell of $K$.
If $u_{i,L}^k>0$, the limit $\eps\to 0$ in the previous estimate leads to a contradiction since $\log u_{i,K}^\eps$ diverges. Therefore, $u_{i,L}^k=0$. Let $L'\in\T$ be a neighboring cell of $L$. Arguing in a similar way as before, it follows that $u_{i,L'}^k=0$. Repeating this argument for all cells in $\T$, we find that $u_{i,K}^k=0$ for all $K\in\T$. This implies that $\sum_{K\in\T}\m(K)u_{i,K}^k=0$ and, by mass conservation, $\sum_{K\in\T}\m(K)u_{i,K}^0=0$, which contradicts the positivity of the $L^1(\Omega)$ norm of $u^0$ in Hypothesis (H1).

\section{Convergence}\label{sec.conv}

In this section, we prove Theorem~\ref{thm.conv}, that is, we show the asserted convergence of the numerical scheme.
Uniform estimates are derived from the entropy inequalities 
\eqref{2.eiB} and~\eqref{2.eiR}. Lemma~\ref{lem.APL} in the appendix
shows that $|\widehat{u}^k|\le\lambda^{-1}|\BB^{1/2}u^k|$, where we 
recall that $\widehat{u}^k=P_{L^\perp}u^k$. Thus, we obtain a uniform estimate
for $\widehat{u}^k$ in the seminorm $|\cdot|_{1,2,\T}$. 
Moreover, since $b_{ii}>0$ and $b_{ij}\ge0$ for all $i,j$ (cf.~Hypothesis (H3)),
estimate \eqref{2.eiR} provides a uniform bound for $u^k$ in the discrete
$L^2(\Omega)$ norm. Hence, there exists a constant $C>0$ 
which is independent of $\eta=\max\{\Delta x,\Delta t\}$ such that
\begin{align}
    \sum_{k=1}^{N_T}\Delta t\big(\|\widehat{u}^k\|_{1,2,\T}^2
	+ \|\BB^{1/2}u^k\|_{1,2,\T}^2\big)
    + \eta^\alpha\sum_{k=1}^{N_T}\Delta t
    |(u^k)^{1/2}|_{1,2,\T}^2 &\le C, \label{4.shannon} \\
	\max_{k=1,\ldots,N_T}\|u^k\|_{0,2,\T}
    + \sum_{j=1}^k\Delta t\sum_{i=1}^n\sum_{\sigma\in\E_{\rm int}}
	\tau_\sigma u_{i,\sigma}^j|\mathrm{D}_{\sigma}(\BB u^j)_i|^2 
    &\le C. \label{4.rao}
\end{align}

\subsection{Compactness properties}

We first prove a full gradient bound with a negative power of $\eta$ on the right-hand side.

\begin{lemma}\label{lem.inv}
There exists $C=C(\zeta)>0$ independent of $\eta$ such that
$$
  \sum_{k=1}^{N}\Delta t|u_i^k|_{1,4/3,\T}^2
	\le C\eta^{-\alpha},\quad 
	 \sum_{k=1}^{N}\Delta t|u_i^k|_{1,1,\T}^2
	\le C\eta^{-\alpha}.
$$
\end{lemma}

\begin{proof}
By the mesh regularity \eqref{2.regul} and property
\eqref{2.sum},
\begin{equation}\label{4.md}
	\sum_{\sigma\in\E_K}\frac{\mm(\sigma)\dist_\sigma}{\m(K)}
	\le \sum_{\sigma\in\E_K}\frac{\mm(\sigma)\dist(x_K,\sigma)}{\zeta\m(K)}
		\le \frac{d}{\zeta}.
\end{equation}
This yields, using H\"older's inequality and the 
$L^2(\Omega)$ bound \eqref{4.rao} for $u_i^k$,
\begin{align*}
	|u_i^k|_{1,4/3,\T}^{4/3} 
	&= \sum_{\sigma\in\E_{\rm int}}\mm(\sigma)\dist_\sigma
	\bigg|\frac{u_{i,L}^k-u_{i,K}^k}{\dist_\sigma}\bigg|^{4/3}\\
	&= \sum_{\sigma\in\E_{\rm int}}
	\mm(\sigma)\dist_\sigma^{-1/3}\big|(u_{i,L}^k)^{1/2}-(u_{i,K}^k)^{1/2}
	\big|^{4/3}\big|(u_{i,L}^k)^{1/2}+(u_{i,K}^k)^{1/2}\big|^{4/3} \\
	&\le \bigg(\sum_{\sigma\in\E_{\rm int}}\mm(\sigma)\dist_\sigma^{-1}
	\big((u_{i,L}^k)^{1/2}-(u_{i,K}^k)^{1/2}\big)^2\bigg)^{2/3} \\
	&\phantom{xx}{}\times
	\bigg(\sum_{\sigma\in\E_{\rm int}}\mm(\sigma)\dist_\sigma
	\big((u_{i,L}^k)^{1/2}+(u_{i,K}^k)^{1/2}\big)^4\bigg)^{1/3} \\
	&\le C|(u_i^k)^{1/2}|_{1,2,\T}^{4/3}\bigg(\sum_{K\in\T}\m(K)
	(u_{i,K}^k)^2\sum_{\sigma\in\E_K}\frac{\mm(\sigma)\dist_\sigma}{\m(K)}
	\bigg)^{1/3} \\
	&\le C(\zeta)|(u_i^k)^{1/2}|_{1,2,\T}^{4/3}\|u_i^k\|_{0,2,\T}^{2/3}.
\end{align*}
Taking the exponent $3/2$, multiplying by $\Delta t$, and summing 
over $k=1,\ldots,N$ proves the first inequality.
The second inequality follows along the same lines (or by H\"older's inequality).
\end{proof}

\begin{lemma}\label{lem.uB}
There exists $C=C(\zeta)>0$ independent of $\eta$ such that
$$
  \sum_{k=1}^{N}\Delta t\|u_{i,\sigma}^k(B\na^{\mathcal{D}}
  \widehat{u}^k)_i\|_{0,4/3,\T}^2 \le C.
$$
\end{lemma}

\begin{proof}
We infer from the definition of the discrete gradient 
and H\"older's inequality that
\begin{align}\label{4.sigma}
  \|&u_{i,\sigma}^k(B\na^{\mathcal{D}}\widehat{u}^k)_i\|_{0,4/3,\T}^{4/3}
  = \sum_{K\in\T}\sum_{\sigma\in\E_{{\rm int},K}}
  \m(T_{K,\sigma})(u_{i,\sigma}^k)^{4/3}
  \bigg|\frac{\mm(\sigma)}{\m(T_{K,\sigma})}\mathrm{D}_{K,\sigma}
  (B\widehat{u}^k)_i\bigg|^{4/3} \\
  &= \sum_{K\in\T}\sum_{\sigma\in\E_{{\rm int},K}}
  \m(T_{K,\sigma})^{1/3}(u_{i,\sigma}^k)^{2/3}
  \frac{\mm(\sigma)^{4/3}}{\m(T_{K,\sigma})^{2/3}}
  \big|(u_{i,\sigma}^k)^{1/2}\mathrm{D}_{K,\sigma}
  (B\widehat{u}^k)_i\big|^{4/3} \nonumber \\
  &\le \bigg(\sum_{K\in\T}\sum_{\sigma\in\E_{{\rm int},K}}
  \m(T_{K,\sigma})(u_{i,\sigma}^k)^2\bigg)^{1/3}
  \bigg(\sum_{K\in\T}\sum_{\sigma\in\E_{{\rm int},K}}
  \frac{\mm(\sigma)^2}{\m(T_{K,\sigma})}u_{i,\sigma}^k
  \big|\mathrm{D}_{K,\sigma}(B\widehat{u}^k)_i\big|^2\bigg)^{2/3}.
  \nonumber 
\end{align}
Because of $\m(T_{K,\sigma})=\mm(\sigma)\dist_\sigma/d$ for $\sigma\in\E_{{\rm int},K}$, mesh regularity \eqref{2.regul}, 
and property \eqref{2.sum}, we find for the first factor that
\begin{align}\label{4.mdu}
  \sum_{K\in\T}\sum_{\sigma\in\E_{{\rm int},K}}
  \m(T_{K,\sigma})(u_{i,\sigma}^k)^2
  &\le C(\zeta)\sum_{K\in\T}
  \bigg(\sum_{\sigma\in\E_{{\rm int},K}}
  \mm(\sigma)\dist(x_K,\sigma)\bigg)(u_{i,K}^k)^2 \\
  &\le C(\zeta)\sum_{K\in\T}\m(K)(u_{i,K}^k)^2
  = C(\zeta)\|u_{i}\|_{0,2,\T}^2, \nonumber
\end{align}
where we also used \eqref{2.meanprop}.
The second factor on the right-hand side of \eqref{4.sigma} becomes
\begin{align*}
  \sum_{K\in\T}\sum_{\sigma\in\E_{{\rm int},K}}
  \frac{\mm(\sigma)^2}{\m(T_{K,\sigma})}u_{i,\sigma}^k
  \big|\mathrm{D}_{K,\sigma}(B\widehat{u}^k)_i\big|^2
  = d\sum_{K\in\T}\sum_{\sigma\in\E_{{\rm int},K}}\tau_\sigma
  u_{i,\sigma}^k\big|\mathrm{D}_{K,\sigma}(B\widehat{u}^k)_i\big|^2.
\end{align*}
We take \eqref{4.sigma} to the power $3/2$, multiply by $\Delta t$, and sum over $k=1,\ldots,N$:
\begin{align*}
  \sum_{k=1}^N\Delta t\|u_{i,\sigma}^k(B\na^{\mathcal{D}}
  \widehat{u}^k)_i\|_{0,4/3,\T}^2
  \le C\max_{k=1,\ldots,N}\|u_i^k\|^2_{0,2,\T}
  \sum_{k=1}^N\Delta t\sum_{\sigma\in\E_{\rm int}}\tau_\sigma  u_{i,\sigma}^k|\mathrm{D}_\sigma(B\widehat{u}^k)_i|^2  \le C,
\end{align*}
where the uniform bound follows from \eqref{4.rao}.
\end{proof}

For the compactness argument, we need an estimate for the discrete time derivative, which is defined by
$$
  \pa_t^{\Delta t} v^k = \frac{v^k-v^{k-1}}{\Delta t} \quad\mbox{for } 
	v\in V_{\T,\Delta t},\ k=1,\ldots,N.
$$

\begin{lemma}[Discrete time derivative]\label{lem.time}
There exists a constant $C=C(\zeta)>0$ independent
of $\eta$ such that 
$$
  \sum_{k=1}^{N}\Delta t\|\pa_t^{\Delta t}u^k\|_{-1,4,\T}^2\le C.
$$
\end{lemma}

 \begin{proof}
	Let $\phi\in V_\T$ be such that $\|\phi\|_{1,4,\T}=1$. We multiply \eqref{2.fvm} by
	$\phi_K$, sum over $K\in\T$, apply discrete integration by parts, and use H\"older's inequality:
	\begin{align*}
		\bigg|\sum_{K\in\T}&\frac{\m(K)}{\Delta t}(u_{i,K}^k-u_{i,K}^{k-1})\phi_K\bigg| \\
		&= \bigg|-\sum_{\sigma\in\E_{\rm int}}\tau_\sigma u_{i,\sigma}^k\mathrm{D}_{K,\sigma}p_i(u^k)\mathrm{D}_{K,\sigma}\phi 
		- \eta^\alpha\sum_{\sigma\in\E_{\rm int}}\tau_\sigma
		\mathrm{D}_{K,\sigma}u_i^k\mathrm{D}_{K,\sigma}\phi\bigg| \\
		&\le C\|u_{i,\sigma}^k(B\na^{\mathcal{D}}\widehat{u}^k)_i
        \|_{0,4/3,\T}|\phi|_{1,4,\T} + \eta^\alpha
        |u_i^k|_{1,4/3,\T}|\phi|_{1,4,\T}.
	\end{align*}
	Then we infer from Lemmas \ref{lem.inv} and \ref{lem.uB} that
	\begin{align*}
	\sum_{k=1}^{N}\Delta t\bigg\|\frac{u_i^k-u_i^{k-1}}{\Delta t}\bigg\|_{-1,4,\T}^2
	\le C(\zeta) + C(\zeta)\eta^\alpha,
	\end{align*}
which concludes the proof.
\end{proof}

The solution $u^k\in V_\T$ to \eqref{2.fvm} refers to a fixed mesh. For each $m\in\mathbb{N}$ let $\T_m$ be a spatial mesh of size $\Delta x_m$ such that the family $\{\T_m\}_{m\in \mathbb{N}}$ satisfies the regularity property~\eqref{2.regul} for a fixed $\zeta>0$ that is independent of $m$. For a time step size $\Delta t_m$, denote by
$\mathcal{D}_m$ the space-time mesh determined by $(\T_m,\Delta t_m)$. 
Let $\Delta x_m$ and $\Delta t_m$ be chosen in such a way that the mesh size $\eta_m=\max\{\Delta x_m,\Delta t_m\}$ of $\mathcal{D}_m$ converges
to zero as $m\to\infty$, and set $N_m=T/\Delta t_m$. Let $u_m=(u_{m,1},\ldots,u_{m,n})$ 
be defined as the piecewise constant function $u_m(x,t) = u_K^k$ for $(x,t)\in K\times[t_{k-1},t_k)$, where $u^k$ is a solution to \eqref{2.fvm} on the mesh $\mathcal{D}_m$, $K\in\T_m$, 
and $k=1,\ldots,N_m$, and set 
$u_m^0=(u_{m,i}^0)_{i=1}^n$, where $u_{m,i}^0(x):=u^0_{i,K}(x)$ for $x\in K$.
Notice that $u^0_m\to u^{\rm in}$ in $L^2(\Om)$  as $m\to\infty$.
Furthermore, we introduce the function 
$u_{m,\sigma}:=(u_{m,i,\sigma})_{i=1}^n$ defined by
$u_{m,i,\sigma}(x,t)=u_{i,\sigma}^k$ for 
$(x,t)\in T_{K,\sigma}\times[t_{k-1},t_k)$, where $K\in\T_m$,
$\sigma\in\E_m$, and $k=1,\ldots,N_m$. This function is piecewise
constant on the dual mesh.

Let $\phi\in V_{\T_m}$ be such that $\|\phi\|_{1,4,\T_m}=1$ and let 
$\widehat{u}_m=P_{L^\perp}u_m$. 
We write $(P_{ij})$ for the matrix associated to $P_{L^\perp}$. Then
\begin{align*}
  \bigg|\sum_{K\in\T_m}\m(K)\pa_t^{\Delta t_m}\widehat{u}_{m,i}|_K\phi_K\bigg|
	&= \bigg|\sum_{K\in\T_m}\sum_{j=1}^n\frac{\m(K)}{\Delta t_m}P_{ij}
	(u_{j,K}^k-u_{j,K}^{k-1})\phi_K\bigg| \\
	&\le C\|\pa_t^{\Delta t_m}u_m^k\|_{-1,4,\T}\|\phi\|_{1,4,\T_m} \le C.
\end{align*}
Together with estimate \eqref{4.shannon}, this implies that
$$
  \sum_{k=1}^{N_m}\Delta t_m\|\pa_t^{\Delta t_m}\widehat{u}^k_m\|_{-1,4,\T_m}^2\le C, \quad
  \sum_{k=1}^{N_m}\Delta t_m\|\widehat{u}^k_m\|_{1,2,\T_m}^2 \le C.
$$
It is shown in \cite[Sec.~6.1]{JuZu21} that the discrete norms $\|\cdot\|_{1,2,\T_m}$
and $\|\cdot\|_{-1,4,\T_m}$ satisfy the assumptions of the compactness result in \cite[Theorem 3.4]{GaLa12}, which we recall in Appendix~\ref{app:compactness} for convenience. 
More precisely, by Proposition~\ref{prop:comp.app}, there exists a subsequence, which is not relabeled, such that
$\widehat{u}_m\to v$ strongly in $L^2(\Omega_T)$
as $m\to\infty$ for some $v\in L^2(\Omega_T)$.
Moreover,  up to a subsequence, we have $u_m\rightharpoonup u$ weakly in $L^2(\Omega_T)$ and consequently $\widehat{u}_m=P_{L^\perp}u_m\rightharpoonup 
P_{L^\perp}u=\widehat{u}$ weakly in
$L^2(\Omega_T)$. This shows that $\widehat{u}=v$.

Estimate \eqref{4.shannon} implies that $y_m:=\na^m\widehat{u}_{m}$ 
is uniformly bounded in $L^2(\Omega_T)$. Hence,
there exists a subsequence (not relabeled) such that
$y_m\rightharpoonup y$ weakly in $L^2(\Omega_T)$. We conclude 
as in \cite[Lemma 4.4]{CLP03} that $y=\na\widehat{u}$. We summarize:
\begin{align}\label{eq:weak.conv.uy}
  u_m\rightharpoonup u, \quad
	y_m \rightharpoonup y=\na\widehat{u} 
    \quad\mbox{weakly in }L^2(\Omega_T).
\end{align}
These convergence results are not sufficient to pass to the limit
in the term $u_{m,i,\sigma}\na^m(\BB u_m)_i$.
The idea is to embed the problem in the larger space of Young measures.
Let $\mathcal{P}(W)$ be the
space of probability measures on $W:=\Rge^n\times (L^\perp)^d$.
Since the sequences $(u_m)$ and $(y_m)$ are bounded in $L^2(\Omega_T)$,
 there exists a subsequence
(not relabeled) and a parametrized probability measure $\mu=(\mu_{x,t})\in  L^\infty_w(\Omega_T;\mathcal{P}(W))$ such that the following holds (cf.~\cite{Ball_1989},~\cite[Theorem 6.2]{Ped97}):
If $f\in C(W)$ and if the sequence $(f(u_m,y_m))$ is equi-integrable,
then its weak limit, which we denote by $\overline{f(u_m,y_m)}$, exists and satisfies
$$
  \overline{f(u_m,y_m)}(x,t) 
 =	\la \mu_{x,t}, f(s,p)\ra\quad\text{ for a.e.\ }(x,t)\in\Omega_T.
$$
In the above reasoning $T\in(0,\infty)$ was arbitrary. Hence, a diagonal argument allows us to choose $\mu$ independent
of $T\in(0,\infty)$ such that 
$\mu\in L_w^\infty(\Omega\times(0,\infty);\mathcal{P}(W))$ and the weak convergences~\eqref{eq:weak.conv.uy} hold for all $T>0$. As a consequence,
$$
  u = \langle\mu,s\rangle, \quad \widehat{u} = \langle\mu,\widehat{s}\rangle, \quad
  y = \langle\mu, p\rangle \quad\mbox{a.e. in }\Omega\times(0,\infty),
$$
where $\widehat{s}=P_{L^\perp}s$. 

\subsection{Convergence of the scheme}

We show that $\mu$ is a dissipative measure-valued solution in the sense of Definition \ref{def.dmvs} satisfying \eqref{2.sepa}. The proof adapts the strategy of \cite{CLP03} to the present situation, where only a weaker form of convergence is known to hold. Let $T\in(0,\infty)$, let $i\in\{1,\ldots,n\}$, 
$\psi\in C_0^\infty(\Omega\times[0,T))$, and let 
$\eta_m=\max\{\Delta x_m,\Delta t_m\}$ be small enough such that
$\operatorname{supp}(\psi)\subset\{x\in\Omega:\dist(x,\pa\Omega)>\eta_m\}
\times[0,T)$. We introduce
\begin{align*}
  F_{10}^m &= -\int_0^T\int_\Omega u_{m,i}\pa_t\psi\dd x\dd t
	- \int_\Omega u_{m,i}^0(x)\psi(x,0)\dd x, \\
	F_{20}^m &= \int_0^T\int_\Omega  u_{m,i,\sigma}\na^m(\BB \widehat{u}_{m})_i\cdot\na\psi\dd x\dd t.
\end{align*}
The convergence results established above imply that, as $m\to\infty$,
$$
  F_{10}^m \to -\int_0^T\int_\Omega u_i\pa_t\psi\dd x\dd t
	- \int_\Omega u_i^{\rm in}(x)\psi(x,0)\dd x.
$$

The limit in $F_{20}^m$ is more involved.
First, Lemma \ref{lem.uB} implies that the term $u_{m,i,\sigma}(\BB\na^m\widehat{u}_m)_i$ is weakly relatively compact in $L^1(\Omega_T)$ and thus weakly convergent in $L^1(\Omega_T)$ along a  subsequence. Second, we assert that
\begin{align}\label{eq:strong.densL1}
	u_{m,\sigma}-u_m\to 0
	\quad\text{in }L^1(\Omega_T)\text{ as }m\to\infty.
\end{align}
We proceed as in \cite[Section 4.2]{Oul18}, but
since we cannot control the full
gradient, we need to rely on the artificial diffusion.
It follows from $\m(T_{K,\sigma})=\dist_\sigma^2\tau_\sigma/d$ that 
\begin{align*}
	\|u_{m,i,\sigma}^k-u_{m,i}^k\|_{0,1,\T_m}
	&\le C \sum_{K\in\T_m}\sum_{\sigma\in\E_{{\rm int},K}}\m(T_{K,\sigma})
	|u_{m,i,\sigma}^k-u_{m,i,K}^k| \\
	&\le C\sum_{K\in\T_m}\sum_{\sigma\in\E_{{\rm int},K}}\m(T_{K,\sigma})
	|u_{m,i,L}^k-u_{m,i,K}^k| \nonumber \\
	&\le C\sum_{\sigma=K|L\in\E_{\rm int}}\m(T_{K,\sigma})
	|u_{m,i,L}^k-u_{m,i,K}^k|\nonumber \\
	&\le C\sum_{\sigma=K|L\in\E_{\rm int}}\dist_\sigma^2
	\tau_\sigma|u_{m,i,L}^k-u_{m,i,K}^k|
	\le C\eta_m|u_{m,i}^k|_{1,1,\T_m}, \nonumber 
\end{align*}
where the constant $C>0$ may change from line to line. 
We take the square, multiply by $\Delta t_m$, sum over $k=1,\ldots,N_m$, and use Lemma \ref{lem.inv}:
\begin{align*}
  \sum_{k=1}^{N_m}\Delta t_m\|u_{m,i,\sigma}^k-u_{m,i}^k\|_{0,1,\T_m}^2
  \le C\eta_m^{2-\alpha}.
\end{align*}
As $m\to+\infty$, the right-hand side goes to zero 
provided that $\alpha<2$. Hence, $u_m-u_{m,\sigma}\to0$ strongly in $L^2(0,T;L^1(\Om))$, which implies~\eqref{eq:strong.densL1}. 
We note that, by interpolation, the strong convergence~\eqref{eq:strong.densL1} together with the fact that the sequence $(u_m-u_{m,\sigma})_m$ is uniformly bounded in $L^2(\Om_T)$ implies that 
\begin{align*}
	u_m-u_{m,\sigma}\to0\quad\text{strongly in }L^p(\Om_T)
    \mbox{ for every }p<2.
\end{align*}

We now assert that, as a consequence of~\eqref{eq:strong.densL1}, the sequence $(u_{m,\sigma},\na^m\widehat{u}_m)$ generates the same Young measure $\mu$ as $(u_m,\na^m\widehat{u}_m)$ (after possibly passing to another subsequence). Indeed, since $\mu$ is uniquely determined by its action on $C_0$-functions, to verify the assertion, it suffices to show that 
\begin{align*}
  \lim_{m\to\infty}\int_{\Omega_T} \big(f(u_{m,\sigma},\na^m\widehat{u}_m)
  - f(u_{m},\na^m\widehat{u}_m)\big)\phi\dd x\dd t =0
\end{align*}
for all $f\in C_0(W)$ and $\phi\in L^1(\Omega_T)$. This follows from~\eqref{eq:strong.densL1} and the dominated convergence theorem, because functions $f\in C_0(W)$ are uniformly continuous.
Since $u_{m,i,\sigma}(\BB\na^m\widehat{u}_m)_i$ is weakly convergent in $L^1(\Omega_T)$, we thus infer that 
\begin{equation*}
	\overline{u_{m,i,\sigma}(\BB\na^m\widehat{u}_m)_i}(x,t)
	= \int_{W}s_i(\BB p)_i\dd \mu_{x,t}(s,p)
    =  \langle\mu_{x,t},s_i(\BB p)_i\rangle.
\end{equation*}
We conclude that
$$
  F_{20}^m \to \int_0^T\int_\Omega\langle\mu_{x,t},
  s_i(\BB p)_i\rangle\dd x\dd t.
$$

Let $\psi_K^k=\psi(x_K,t_k)$ and multiply \eqref{2.fvm} by
$\Delta t_m \psi_K^{k-1}$ and sum over $K\in\T_m$, $k=1,\ldots,N_m$. 
This gives $F_1^m+F_2^m+F_3^m=0$, where
\begin{align*}
  F_1^m &=  \sum_{k=1}^{N_m}\sum_{K\in\T_m}
	\m(K)(u_{i,K}^k-u_{i,K}^{k-1})\psi_K^{k-1}, \\
	F_2^m &= -\sum_{k=1}^{N_m}\Delta t_m\sum_{K\in\T_m}\sum_{\sigma\in\E_{{\rm int},K}}
	\tau_\sigma u_{i,\sigma}^k\mathrm{D}_{K,\sigma}(\BB u^k)_i\psi_K^{k-1}, \\
	F_3^m &= -\eta_m^\alpha\sum_{k=1}^{N_m}\Delta t_m\sum_{K\in\T_m}
	\sum_{\sigma\in\E_{{\rm int},K}}\tau_\sigma\mathrm{D}_{K,\sigma}u_i^k\psi_K^{k-1}.
\end{align*}
We infer from the Cauchy--Schwarz inequality and Lemma \ref{lem.inv} that
\begin{align*}
  |F_3^m|\le \eta_m^\alpha\bigg(\sum_{k=1}^{N_m}\Delta t_m |u_i^k|_{1,4/3,\T_m}^2\bigg)^{1/2}\bigg(\sum_{k=1}^{N_m}\Delta t_m
  |\psi^{k-1}|_{1,4,\T_m}^2\bigg)^{1/2}
  \le C\eta_m^{\alpha/2}\to 0
\end{align*}
as $m\to\infty$. We claim that $F_{j0}^m-F_j^m\to 0$ for $j=1,2$. 

For the limit of $F_{10}^m-F_1^m$, we use as in the proof of \cite[Theorem 5.2]{CLP03}
discrete integration by parts in time:
\begin{align*}
  F_{1}^m &= -\sum_{k=1}^{N_m}\sum_{K\in\T_m}
	\m(K)u_{i,K}^k(\psi_K^k-\psi_K^{k-1})
	- \sum_{K\in\T_m}\m(K)u_{i,K}^0\psi_K^0 \\
	&= -\sum_{k=1}^{N_m}\sum_{K\in\T_m}\int_{t_{k-1}}^{t_k}\int_K u_{i,K}^k
	\pa_t\psi(x_K,t)\dd x\dd t -\sum_{K\in\T_m}\int_K u^0_{i,K}\psi(x_K,0)\dd x, \\
	F_{10}^m &= -\sum_{k=1}^{N_m}\sum_{K\in\T_m}\int_{t_{k-1}}^{t_k}
    \int_K u_{i,K}^k\pa_t\psi(x,t)\dd x\dd t 
    -\sum_{K\in\T_m}\int_K u^0_{i,K}\psi(x,0)\dd x.
\end{align*}
It follows from the regularity of $\psi$ that
$$
  |F_{10}^m-F_1^m| \le C(\Omega_T)\|u_{i}^k\|_{L^\infty(0,T;L^2(\Omega))}
	\|\psi\|_{C^2(\overline{\Omega}_T)}\Delta t_m\to 0\quad\mbox{as }m\to\infty.
$$

We deduce from the definition of the discrete gradient that
\begin{align*}
  F_{20}^m &= \sum_{k=1}^{N_m}\int_{t_k}^{t_{k-1}}
  \sum_{\sigma\in\E_{\rm int}}
  \frac{\mm(\sigma)}{\m(T_{K,\sigma})}u_{i,\sigma}^k
  \mathrm{D}_{K,\sigma}(\BB\widehat{u}_m)_i\int_{T_{K,\sigma}}
  \na\psi\cdot\nu_{K,\sigma}\dd x\dd t, \\
  F_2^m &= \sum_{k=1}^{N_m}\int_{t_k}^{t_{k-1}}
  \sum_{\sigma\in\E_{\rm int}}\frac{\mm(\sigma)}{\dist_\sigma}
  u_{i,\sigma}^k\mathrm{D}_{K,\sigma}(\BB\widehat{u}_m)_i
  \mathrm{D}_{K,\sigma}\psi^{k-1}\dd t.
\end{align*}
This gives
\begin{align*}
  |F_{20}^m-F_{2}^m| 
  &\le \sum_{k=1}^{N_m}\sum_{\sigma\in\E_{\rm int}}
  \mm(\sigma)u_{i,\sigma}^k
  |\mathrm{D}_{K,\sigma}(\BB\widehat{u}_m^k)_i| \\
  &\phantom{xx}{}\times\bigg|\int_{t_{k-1}}^{t_k}\bigg(
  \frac{\mathrm{D}_{K,\sigma}\psi^{k-1}}{\dist_\sigma} 
  - \frac{1}{\m(T_{K,\sigma})}
  \int_{T_{K,\sigma}}\na\psi\cdot\nu_{K,\sigma} \dd x\bigg)\dd t\bigg|.
\end{align*}
By the proof of Theorem 5.1 in \cite{CLP03}, there exists $C>0$, independent of $\eta_m$, such that
$$
  \bigg|\int_{t_{k-1}}^{t_k}\bigg(
	\frac{\mathrm{D}_{K,\sigma}\psi^{k-1}}{\dist_\sigma} - \frac{1}{\m(T_{K,\sigma})}
	\int_{T_{K,\sigma}}\na\psi\cdot\nu_{K,\sigma}\dd x\bigg)\dd t\bigg|
	\le C\Delta t_m\eta_m,
$$
which shows, using the Cauchy--Schwarz inequality, that
\begin{align*}
  |F_{20}^m-F_{2}^m| 
  &\le C\eta_m \sum_{k=1}^{N_m}\Delta t_m\sum_{\sigma\in\E_{\rm int}}
  \mm(\sigma)u_{i,\sigma}^k|\mathrm{D}_\sigma(\BB\widehat{u}_m)_i| \\
  &\le C\eta_m\sum_{k=1}^{N_m}\Delta t_m|(\BB u_m^k)_i|_{1,2,\T_m}
  \bigg(\sum_{K\in\T_m}\sum_{\sigma\in\E_{{\rm int},K}}\mm(\sigma)
\dist_\sigma(u_{i,\sigma}^k)^2\bigg)^{1/2}.
\end{align*}
We conclude from the Cauchy--Schwarz inequality, estimate \eqref{4.mdu}, and the uniform bounds \eqref{4.shannon}--\eqref{4.rao} that
\begin{align*}
  |F_{20}^m-F_{2}^m| &\le C(\zeta)\eta_m
	\bigg(\sum_{k=1}^{N_m}\Delta t_m|(\BB u^k)_i|_{1,2,\T_m}^2\bigg)^{1/2}
	\bigg(\sum_{k=1}^{N_m}\Delta t_m\|u_i^k\|_{0,2,\T_m}^2\bigg)^{1/2} \\
	&\le C(\zeta)\eta_m\to 0 \quad\mbox{as }m\to\infty.
\end{align*}
We deduce that $F_{10}^m+F_{20}^m\to 0$ as $m\to\infty$.
Then, because of $F_1^m+F_2^m+F_3^m=0$, 
$$
  F_{10}^m+F_{20}^m = (F_{10}^m-F_1^m) + (F_{20}^m-F_2^m) - F_3^m \to 0 
	\quad\mbox{as }m\to\infty,
$$
which proves that $u_i$ satisfies
$$
  \int_0^T\int_\Omega u_i\pa_t\psi\dd x\dd t + \int_\Omega u_i^{\rm in}\psi(0)\dd x
  = \int_0^T\int_\Omega\langle\mu_{x,t},s_i(\BB p)_i\rangle\cdot 
  \na\psi\dd x\dd t.
$$
Hence, in the sense of distributions,
\begin{equation}\label{4.du}
  \pa_t u_i = \diver\langle\mu,s_i(\BB p)_i\rangle, \quad
  u_i(0) = u_i^{\rm in}, \quad i=1,\ldots,n.
\end{equation}

\subsection{Entropy inequalities}\label{sec.ent.ineq}

We verify the entropy inequalities
\eqref{1.eiBm} and \eqref{1.eiRm}.
The definition of $u^0_m$ and the regularity $u^{\rm in}\in L^2(\Om)$ imply the strong convergence $u^0_m\to u^{\rm in}$ in $L^2(\Om)$  as $m\to\infty$.

{\em Re Shannon:}
Since  $(u_m)_m$ is bounded in $L^2(\Om_T)$, the sequence $(h_S(u_m))_m\subset L^1(\Om_T)$ is equi-integrable.
After passing to a subsequence, we can therefore assume that $(h_S(u_m))_m$ is weakly convergent in $L^1(\Om_T)$, which implies that for a.e.~$(x,t)\in\Om_T$,
$$
\la\mu_{x,t},h_S(s)\ra = \overline{h_S(u_m)}(x,t).
$$
The dual mesh allows us to rewrite the
\SShannon entropy dissipation in \eqref{2.eiB}
as
$$
\sum_{j=1}^k\Delta t_m|\BB^{1/2}u^j_m|^2_{1,2,\T_m}
= \int_0^{t_k}\int_\Omega|\na^m(\BB^{1/2}u_m)|^2\dd x\dd \tau.
$$
Given $0<\delta\ll1$, let $m$ be large enough such that $\Delta t_m<\delta$. 
Then~\eqref{2.eiB}
entails for all $t\in[\delta,T]$ that
\begin{align}\label{eq:Shndisc}
	H_\Shn(u_m(t)) + \int_0^{t-\delta}\int_\Omega|\na^m(\BB^{1/2}u_m)|^2\dd x\dd \tau
	&\le H_\Shn(u^0_m). 
\end{align}
Next, let $\xi\in C^1_c([0,T);\mathbb{R}_\ge)$ with $\xi(0)=1$ and $\xi'\le0$.
We multiply the last inequality by the nonnegative function $-\xi'(t)$
and integrate over $t\in[\delta,T]:$ 
\begin{align*}
		\int_{\delta}^T\int_\Omega (-\xi'(t))
		h_\Shn(u_m(t))\dd x\dd t
+\int_\delta^T(-\xi'(t))
\int_0^{t-\delta}\!\!\int_\Omega|\na^m(\BB^{1/2}\widehat{u}_m)|^2\dd x
\dd \tau\dd t
	\le \xi(\delta) \,H_\Shn(u^0_m).
\end{align*}
We take the $\liminf_{m\to\infty}$ in the above inequality, where we invoke~\cite[Theorem 6.11]{Ped97} for the second term on the left-hand side.
This yields
\begin{align*}
	\int_{\delta}^T(-\xi'(t))\int_\Omega 
	\la \mu_{x,t},h_\Shn(s)\ra\dd x\dd t
	+\int_\delta^T(-\xi'(t))
	\int_0^{t-\delta}\int_\Omega\la\mu_{x,\tau},|\BB^{1/2}p|^2\ra\dd x\dd \tau\dd t
	\le \xi(\delta)H_\Shn(u^{\rm in}).
\end{align*}
As $\delta\downarrow0$, we infer
\begin{align*}
	\int_0^T(-\xi'(t))\int_\Omega 
	\la \mu_{x,t},h_\Shn(s)\ra\dd x\dd t
	+\int_0^T(-\xi'(t))
	\int_0^t\int_\Omega\la\mu_{x,\tau},|\BB^{1/2}p|^2\ra\dd x\dd \tau\dd t
	\le H_\Shn(u^{\rm in}).
\end{align*}
This is true for all $\xi\in C^1_c([0,T);\mathbb{R}_\ge)$ with $\xi(0)=1$ and $\xi'\le0$. We then choose $\xi=\xi_\ell$ with $(\xi_\ell)_\ell$ a suitable approximation of the Heaviside-type function $1_{[0,t_0]}$ and let $\ell\to\infty$ to deduce~\eqref{1.eiBm} at time $t=t_0$ for a.e.\ $t_0\in(0,T]$.
\medskip

{\em Re Rao:}
Next, we verify~\eqref{1.eiRm} and the time monotonicity of $H_R(u)$. Since $(\widehat{u}_m)$ converges strongly to $\widehat{u}$ in $L^2(\Omega_T)$, we find that
$$
  H_R(u(t)) = \frac12\int_\Omega|\BB^{1/2}\widehat{u}(t)|^2\dd x
  = \frac12\lim_{m\to\infty}\sum_{K\in\T_m}\m(K)
  |\BB^{1/2}\widehat{u}_m(t)|^2 = \lim_{m\to\infty}H_R(u_m(t)).
$$
Together with the non-increase of $[0,\infty)\ni t\mapsto H_R(u_m(t))$ (cf.~Theorem~\ref{thm.ex}), this implies 
that the mapping $t\mapsto H_R(u(t))$ is nonincreasing.
It remains to show~\eqref{1.eiRm}. 
To this end, we let $0<\delta\ll 1$ and take $m$ large enough so that $\Delta t_m<\delta$.
Then it follows from the discrete Rao entropy inequality~\eqref{2.eiR} that
$$
	H_R(u_m(t)) + \sum_{i=1}^n\int_0^{t-\delta}\int_\Omega u_{m,i,\sigma}|(\BB\na^m\widehat{u}_m)_i|^2\dd x\dd \tau
\le H_R(u^0_m).
$$
To estimate below the $\liminf_{m\to\infty}$ of the second term on the left-hand side, we recall  that $\mu$ is also the Young measure associated with 
$(u_{m,\sigma},\na^m\widehat{u}_m)$. We therefore infer from~\cite[Theorem 6.11]{Ped97} for every $i\in\{1,\dots,n\}$
\begin{align*}
\int_0^{t-\delta}\int_\Omega \la\mu_{x,\tau},s_i|(\BB p)_i|^2\ra\dd x\dd \tau\le 	\liminf_{m\to\infty}\int_0^{t-\delta}\int_\Omega u_{m,i,\sigma}|(\BB\na^m\widehat{u}_m)_i|^2\dd x\dd \tau.
\end{align*}
Thus, in the limit $m\to\infty$ we deduce
$$
H_R(u(t)) + \sum_{i=1}^n
\int_0^{t-\delta}\int_\Omega \la\mu_{x,\tau},s_i|(\BB p)_i|^2\ra\dd x\dd \tau
\le H_R(u^{\rm in}),
$$
and sending $\delta\downarrow0$ we obtain~\eqref{1.eiRm}.

\subsection{Separation of the \texorpdfstring{$\widehat{s}$-}{}component} 

For simplicity, we only prove identity~\eqref{2.sepa} in the case where $f=f(s)\in C_0(\Rge^n)$. 
Let $g(s_1,s_2)=f(s_1+s_2)$, defined on the convex set
$$
  Q = \{(s_1,s_2)\in L^\perp\times L:s_1+s_2\in\Rge^n\}. 
$$
Since the sequence $(\widehat{u}_m)$ converges strongly in $L^2(\Omega_T)$,
the Young measure $\widetilde{\mu}$, generated by $(P_{L^\perp}u_m,P_Lu_m)$,
has the form $\widetilde{\mu}_{x,t}=\delta_{\widehat{u}(x,t)}\otimes\nu_{x,t}$,
where $\nu=(\nu_{x,t})$ is the Young measure generated by the sequence
$(P_Lu_m)$ \cite[Prop.~6.13]{Ped97}. Hence, by construction of
$\mu$ and $\widetilde{\mu}$,
\begin{align*}
	\int_{\Rge^n}f(s)\dd \mu_{x,t}(s)
	&= \int_Q g(s_1,s_2)\dd\widetilde{\mu}_{x,t}(s_1,s_2)
	= \int_Q g(\widehat{u}(x,t),s_2)\dd \widetilde{\mu}_{x,t}(s_1,s_2) \\
	&= \int_Q f(\widehat{u}(x,t)+s_2)\dd \widetilde{\mu}_{x,t}(s_1,s_2).
\end{align*}
 It follows that $\la\mu_{x,t},f(s)\ra=\la {\widetilde{\mu}_{x,t}}, f(\widehat{u}(x,t)+s_2)\ra$
	for all $f=f(s)\in C_0(\Rge^n)$ and a.a.\ $(x,t)$.

\subsection{Time regularity}\label{ssec:tempreg}

The time regularity for the density part $u=\la\mu,s\ra$ of 
the barycenter of $\mu$ follows from the continuity equation~\eqref{4.du}. To see this, 
we first note that due to $b_{ii}>0, b_{ij}\ge0$, and 
 property~\eqref{2.sepa},
\begin{align}\label{eq:L2.mv}
	\bigg\langle\mu_{x,t},\sum_{i=1}^n s_i^2\bigg\rangle
    \le C\big\langle\mu_{x,t},|B^{1/2}\widehat{s}|^2\big\rangle
    = C|B^{1/2}\widehat u(x,t)|^2 = Ch_R(u(x,t))
\end{align}
 for a.e.\ $(x,t)\in\Om\times(0,\infty)$.
Then we use Jensen's inequality to estimate for $i=1,\dots,n$,
\begin{align*}
	\|\la\mu&,s_i(B\qbf)_i\ra\|_{L^2(0,\infty;L^{4/3}(\Om))}^2
	\le\int_0^\infty \bigg(\int_\Om\la\mu_{x,t},|s_i(B\qbf)_i|^{4/3}\ra\dd x
	\bigg)^{3/2}\dd t \\
	&\le\int_0^\infty \bigg(\int_\Om\langle\mu_{x,t},s_i^2\rangle^{1/3}
	\langle\mu_{x,t},s_i|(B\qbf)_i|^2\rangle^{2/3}\dd x
	\bigg)^{3/2}\dd t \\
	&\le \int_0^\infty \bigg(\int_\Om\langle\mu_{x,t},s_i^2\rangle\dd x\bigg)^{1/2}
	\int_\Om\langle\mu_{x,t},s_i|(B\qbf)_i|^2\rangle\dd x
    \dd t \\
	&\le \bigg(\esssup_{0<t<\infty}
	\int_\Om\langle\mu_{x,t},s_i^2\rangle\dd x\bigg)^{1/2}
    \bigg(\int_0^\infty\int_\Om\langle\mu_{x,t},s_i|(B\qbf)_i|^2
    \rangle\dd x\dd t\bigg),
\end{align*}
where H\"older's inequality was applied several times.
It therefore follows from~\eqref{4.du} that 
\begin{align*}
		\|\partial_tu_i\|_{L^2(0,\infty;W^{1,4}(\Om)^*)}
		\le 	\|\la\mu,s_i(B\qbf)_i\ra\|_{L^2(0,\infty;L^{4/3}(\Om))}
		\le CH_R(u^{\rm in}),
\end{align*}
where the last step also uses~\eqref{1.eiRm} and~\eqref{eq:L2.mv}.
This finishes the proof of Theorem~\ref{thm.conv}. 

\begin{remark}[Curved domains]\label{rem.curved}\rm
We claim that Theorems \ref{thm.ex} and \ref{thm.conv} also hold for curved Lipschitz domains $\Omega\subset\R^d$. 
The triangulation then contains control volumes with curved segments that are part of $\pa\Omega$. The analysis of this section is still possible, since we consider no-flux boundary conditions and no boundary values need to be defined. 
 The analysis has to be adapted in two points. First, the convergence of the scheme is typically proved on polygonal meshes and the error between the curved cell and the polygonal cell (which is of order $(\Delta x)^{d+1}$) needs to be taken into account. Second, as the compactness of the approximate sequence has been established for polygonal domains \cite{GaLa12}, the error between the approximate sequence and its extension by zero to the polygonal domain has to be estimated. In two space dimensions, it is of order $\Delta x$; see \cite[Prop.~4.14]{Nab16} for details.
The drawback of this approach is that one has to perform numerical integrations over the curved elements, which may be cumbersome in particular in three space dimensions. 
\begin{figure}[ht]
	\includegraphics[height=45mm,width=65mm]{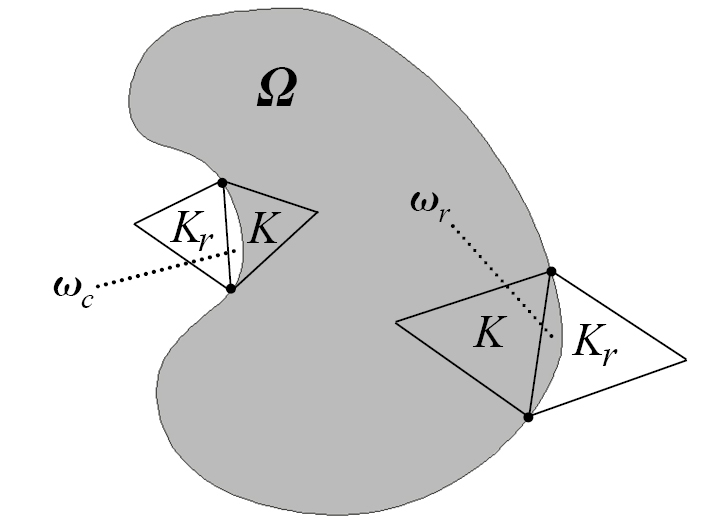}
	\caption{Triangulation of a curved domain.}
	\label{fig.domain}
\end{figure}

Here we report on the simple approach of \cite{ElJa83}. The idea is to cover $\Omega$ by additional control volumes and to estimate the integral error. To simplify the presentation, let $\Omega\subset\R^2$ and let $\T$ be a sufficiently fine triangulation of $\Omega$ into triangles. 
To each cell with two vertices on $\pa\Omega$, we add the reflected triangle to the triangulation such that $\Omega\subset\cup_{K\in\T^*}K$, where $\T^*$ consists of all cells $K\in\T$ and the associated reflected cells $K_r$ with nonempty intersection with $\Omega$; see Figure \ref{fig.domain}. Denoting by $\omega_r=K_r\cap\Omega$ if $K_r\cap\Omega\neq\emptyset$ and $\omega_c=K\setminus\Omega$ if $K_r\cap\Omega=\emptyset$, the domain splits into 
$$
  \Omega = \Omega_h \cup \Omega_r \setminus \Omega_c
  := \bigg(\bigcup_{K\in\T}K\bigg)\cup
  \bigg(\bigcup_{\omega_r}\omega_r\bigg)\setminus
  \bigg(\bigcup_{\omega_c}\omega_c\bigg).
$$
We can perform the numerical analysis on $V_{\T^*}$ as in Sections \ref{sec.ex} and \ref{sec.conv}. For the convergence of the scheme, we need to show that the difference of the integrals over $\Omega_h$ and $\Omega$ vanishes when $\eta_m\to 0$. The difference consists of two contributions: the integral over $\Omega_r$ and the integral over $\Omega_c$. We illustrate the convergence for the integral
\begin{align*}
  \bigg|\int_{\Omega_r}u_{m,i,\sigma}\na^m(\BB\widehat{u}_m)_i
  \cdot\na\psi\dd x\bigg|
  \le C\sum_{\omega_r}\m(\omega_r)\|u_{m,i,\sigma}\|_{0,\infty,\omega_r}
  \|\na^m(\BB\widehat{u}_m)_i\|_{0,\infty,\omega_r},
\end{align*}
where $\psi$ is a smooth test function.
 By the inverse inequality \cite[Section 21.1]{ErGu21}
$$
  \|v\|_{0,\infty,\omega_r}
  \le \|v\|_{0,\infty,K_r}
  \le C(\Delta x)^{-d/2}\|v\|_{0,2,K_r},
$$
the bound $\m(\omega_r)\le C(\Delta x)^{d+1}$ (which is valid under certain regularity conditions on the mesh), and the Cauchy--Schwarz inequality, we have
\begin{align*}
  \bigg|\int_{\Omega_r}u_{m,i,\sigma}\na^m(\BB\widehat{u}_m)_i
  \cdot\na\psi\dd x\bigg|
  &\le C\Delta x
  \bigg(\sum_{K_r}\|u_{m,i,\sigma}\|_{0,2,K_r}^2\bigg)^{1/2}
  \bigg(\sum_{K_r}\|\na^m(\BB\widehat{u}_m)_i\|_{0,2,K_r}^2
  \bigg)^{1/2} \\
  &\le C\Delta x\to 0\quad\mbox{as }\eta\to 0,
\end{align*}
taking into account the uniform bounds from \eqref{2.eiB} and \eqref{2.eiR}. In a similar way, the integral over $\Omega_c$ tends to zero as $\eta\to 0$.
\end{remark}

\section{Stability}

In this section, we prove Theorem~\ref{thm.wsu}.
Let $\mu$ be a dissipative measure-valued solution and let $v\in C^1(\overline{\Omega}_T)$ be a positive 
solution of \eqref{1.eq}, \eqref{1.bic}.
We introduce the relative  \SShannon and Rao entropies by, respectively,
\begin{align*}
	H_\Shn^{\rm mv}(u(t)|v(t)) 
	&= \sum_{i=1}^n\int_\Omega \big(\langle\mu_{x,t},\shn(s_i)\rangle
	- \shn(v_i(x,t)) - \shn'(v_i(x,t))\cdot(u_i-v_i)(x,t)\big) \dd x, \\
	&= \int_\Omega \sum_{i=1}^n\big(\langle\mu_{x,t},s_i\log s_i\rangle
	- u_i\log v_i-(u_i-v_i)\big) \dd x\ge0, \\
	H_R(u(t)|v(t)) &= \frac12\int_\Omega|\BB^{1/2}(u-v)(x,t)|^2\dd x\ge 0,
\end{align*}
where $\shn(z)=z(\log z-1)+1$ for $z\ge 0$. 
We further define the usual relative  \SShannon  entropy 
$H_\Shn(u|v)=\int_\Omega \sum_{i=1}^n\big(u_i\log u_i
- u_i\log v_i-(u_i-v_i)\big) \dd x.$
Furthermore, we set
\begin{align*}
	H_{\rm rel}^{\rm mv}(u|v) = H_\Shn^{\rm mv}(u|v) + H_R(u|v),
	\\H_{\rm rel}(u|v) = H_\Shn(u|v) + H_R(u|v).
\end{align*}
We first compute the relative entropy inequalities.

\begin{lemma}[Relative entropy inequalities]
Suppose that $\Omega$ has a Lipschitz boundary.
Let $\mu$ be a dissipative measure-valued solution, $u:=\la\mu,s\ra$, and let $v\in C^1(\overline{\Omega}_T)$  be a positive solution to \eqref{1.eq}, \eqref{1.bic} for $t\in(0,T)$ (in the weak sense). 
Then, for a.e.\ $t\in(0,T)$,
\begin{align}
    &H_\Shn^{\rm mv}(u(t)|v(t)) 
	+ \int_0^t\int_\Omega\big\langle\mu_{x,\tau},
	|\BB^{1/2}(p-\na v)|^2\big\rangle\dd x\dd \tau \label{4.HB} \\
	&\phantom{x}{}+ \int_0^t\int_\Omega\bigg\langle\mu_{x,\tau},\sum_{i=1}^n
	(v_i-s_i)\na \log v_i \cdot(\BB(p-\na v))_i\bigg\rangle\dd x\dd \tau
    \le H_\Shn(u^{\rm in}|v(0)), \nonumber \\
	&H_R(u(t)|v(t)) 
	 +\int_0^t\int_\Omega \sum_{i=1}^n\la\mu_{x,\tau},s_i|(B(p-\nabla v))_i|^2\ra\dd x\dd \tau \label{4.HR} \\
	&\phantom{x}{}+ \int_0^t\int_\Omega\bigg\langle\mu_{x,\tau},
	\sum_{i=1}^n(s_i-v_i)\na(\BB v)_i\cdot(\BB(p-\na v))_i
	\bigg\rangle\dd x\dd \tau \le H_R(u^{\rm in}|v(0)).\nonumber
\end{align}
\end{lemma}

\begin{proof}
It follows from~\eqref{1.evol} 
that for all $i=1,\ldots,n$ and
 $\phi\in L^2(0,T;W^{1,4}(\Omega))$ 
\begin{equation}\label{4.evol}
\int_0^T(\pa_t u_i,\phi)_{W^{1,4}(\Omega)^*}\dd t
	= -\int_0^T\int_\Omega\langle\mu_{x,t},s_i(\BB p)_i
	\rangle\cdot\na\phi \dd x\dd t,
\end{equation}
where $(\cdot,\cdot)_{W^{1,4}(\Omega)^*}$ denotes the duality pairing
between $W^{1,4}(\Omega)^*$ and $W^{1,4}(\Omega)$.

{\em Re Shannon:} 
The solution property and positivity of $v$ imply that for every $\psi\in C^1(\overline\Om_T;\R^n)$, 
\begin{align*}
	-\sum_{i=1}^n\int_\Om (\partial_t\log v_i) \psi_i\dd x
	&=	\int_\Om \sum_{i=1}^n v_i\nabla (Bv)_i\cdot\nabla \bigg(\frac{\psi_i}{v_i}\bigg)\dd x \\
	&=	\int_\Om \nabla v:\nabla (B\widehat \psi)\dd x
	-\sum_{i=1}^n\int_\Om \nabla (Bv)_i\cdot(\na\log v_i)\psi_i\dd x.
\end{align*}
Let $t\in(0,T)$ be arbitrary. An integration over  $\tau\in(0,t)$ and an approximation argument imply that for all $\psi\in L^2(\Om_T;\R^n)$ with $\nabla B\widehat\psi\in L^2(\Om_T)$,
\begin{align*}
	-\sum_{i=1}^n\int_0^t\int_\Om (\partial_t \log v_i)\psi_i\dd x\dd\tau
	&=	\int_0^t\int_\Om \nabla v:\nabla B\widehat\psi\dd x\dd\tau \\
	&\phantom{xx}{}-\sum_{i=1}^n\int_0^t\int_\Om \nabla (Bv)_i\cdot(\na\log v_i) \,\psi_i\dd x\dd\tau.
\end{align*}
The choice $\psi=u=\la\mu,s\ra$ and the property
$\na B\widehat{u}=B\langle\mu,p\rangle=\langle\mu,Bp\rangle$ lead to
\begin{align*}
	-\sum_{i=1}^n&\int_0^t\int_\Om (\partial_t\log v_i)u_i\dd x\dd\tau= \\
	&= \int_0^t\int_\Om \nabla v:\nabla B\widehat u\dd x\dd\tau
	- \sum_{i=1}^n\int_0^t\int_\Om \na (Bv)_i\cdot(\na\log v_i) u_i\dd x\dd\tau \\
	&=	\int_0^t\int_\Om \la\mu_{x,\tau},B^{1/2}\nabla v:B^{1/2}p\ra\dd x\dd\tau
	-	\sum_{i=1}^n\int_0^t\int_\Om \la\mu_{x,\tau}, s_i\nabla\log v_i \cdot \nabla (Bv)_i\ra\dd x\dd\tau.
\end{align*}
Next, we use $\phi_i=1_{[0,t]}\log v_i$ as a test function in the weak formulation \eqref{4.evol}, multiply by $-1$, and sum over $i=1,\dots ,n$:
\begin{align*}
	-\sum_{i=1}^n\int_0^t( \partial_t u_i,\log v_i )_{W^{1,4}(\Om)^*}\dd\tau
	= \sum_{i=1}^n\int_0^t\int_\Omega\langle\mu_{x,\tau},s_i(\BB p)_i
	\rangle\cdot\na\log v_i  \dd x\dd \tau.
\end{align*}
We add the previous two equations:
\begin{align*}
	-\int_0^t\frac{\dd}{\dd t}\int_\Om\sum_{i=1}^n(\log v_i)u_i\dd x\dd\tau
	&=	\int_0^t\int_\Om \la\mu_{x,\tau},B^{1/2}\nabla v:B^{1/2}p\ra\dd x\dd\tau \\
	&\phantom{xx}{}+\int_0^t\int_\Om \la\mu_{x,\tau}, \sum_{i=1}^ns_i\nabla\log v_i \cdot  (B(p-\nabla v))_i\ra
    \dd x\dd\tau.
\end{align*}
Combined with the identity
\begin{align*}
	\int_0^t\int_\Om &\la\mu_{x,\tau},B^{1/2}\nabla v:B^{1/2}p\ra
    \dd x\dd\tau
	-\int_0^t\int_\Om \la\mu_{x,\tau},|B^{1/2}\nabla v|^2\ra\dd x\dd\tau\\
    &-\int_0^t\int_\Om \bigg\la\mu_{x,\tau}, \sum_{i=1}^nv_i\nabla\log v_i \cdot  (B(p-\nabla v))_i\bigg\ra\dd x\dd\tau
	=0,
\end{align*}
the Shannon entropy inequality~\eqref{1.eiBm}, and mass conservation $ (\dd/\dd t)\int_\Om v_i\,\dd x=0$,  this gives~\eqref{4.HB}.

{\em Re Rao:}
Since $v_i\nabla (Bv)_i\in L^2(\Omega_T)$, we can test the equation for $v$ with $1_{[0,t]}\,B(v-u)\in L^2(0,T;H^1(\Omega))$. This yields
\begin{align*}
	\int_0^t\int_\Omega\partial_t v^TB(v-u)\dd x\dd \tau 
	= -\int_0^t\int_\Omega \sum_{i=1}^nv_i\nabla(Bv)_i\cdot \nabla (B(v-u))_i\dd x\dd \tau. 
\end{align*}
Next, we choose $\phi=1_{[0,t]}(Bv)_i$ in equation~\eqref{4.evol} for $u$ and sum over $i=1,\dots,n$:
\begin{align*}
	-\int_0^t(\partial_t u,Bv)_{W^{1,4}(\Omega)^*}\dd \tau 
	= \int_0^t\int_\Omega \sum_{i=1}^n\la\mu_{x,\tau},s_i(Bp)_i\ra\cdot  (B\nabla v)_i\dd x\dd \tau.
\end{align*}
Adding to these identities the Rao entropy inequality~\eqref{1.eiRm} and rearranging terms gives
\begin{align*}
		\frac{1}{2}\int_0^t&\frac{\dd}{\dd t}\int_\Omega
		(u-v)^TB(u-v)\dd x\dd \tau 
		\le-\int_0^t\int_\Omega \sum_{i=1}^n\la\mu_{x,\tau},s_i|(B(p-\nabla v))_i|^2\ra\dd x\dd \tau \\
        &{}- \int_0^t\int_\Omega \sum_{i=1}^n\la\mu_{x,\tau},(s_i-v_i)(B\nabla v)_i\cdot  (B(p-\nabla v))_i\dd x\dd \tau, 
\end{align*}
which implies~\eqref{4.HR}, concluding the proof.
\end{proof}

We proceed with the proof of Theorem \ref{thm.wsu}. To this end,
we estimate the last integrals on the left-hand sides of \eqref{4.HB} and
\eqref{4.HR}. We infer from Young's inequality that
\begin{align}
  \bigg|\sum_{i=1}^n&(v_i-s_i)\na \log v_i 
	\cdot(\BB(p-\na v))_i\bigg| \label{4.aux1} \\
	&\le \frac14|\BB^{1/2}(p-\na v)|^2 + C\sum_{i=1}^n |\na\log v_i|^2(s_i-v_i)^2
	\nonumber \\
	&\le \frac14|\BB^{1/2}(p-\na v)|^2 + C|s-v|^2, \nonumber \\
	\bigg|\sum_{i=1}^n&(s_i-v_i)\na(\BB v)_i\cdot
	(\BB(p-\na v))_i\bigg| \label{4.aux2} \\
	&\le \frac14|\BB^{1/2}(p-\na v)|^2 + C\sum_{i=1}^n|\na(\BB v)_i|^2
	(s_i-v_i)^2 \nonumber \\
	&\le \frac14|\BB^{1/2}(p-\na v)|^2 + C|s-v|^2, \nonumber
\end{align}
where $C>0$ depends on the $L^\infty(\Omega_T)$ norms of 
$|\na\log v_i|$ and $\na(\BB v)_i$.
Thus, adding the relative entropy inequalities \eqref{4.HB} and \eqref{4.HR},
the first terms on the right-hand sides of \eqref{4.aux1} and \eqref{4.aux2}
can be absorbed by the left-hand side of \eqref{4.HB} such that
\begin{align}\label{eq:ediss-rel}\begin{aligned}
   H_{\rm rel}^{\rm mv}(u(t)|v(&t)) + \int_0^t\int_\Omega\bigg\langle\mu_{x,\tau},
	\frac12|\BB^{1/2}(p-\na v)|^2 
	\bigg\rangle\dd x\dd \tau \\
	&\le C\int_0^t\int_\Omega\langle\mu_{x,\tau},|s-v|^2\rangle\dd x\dd \tau
	+ H_{\rm rel}(u^{\rm in}|v(0)). 
\end{aligned}
\end{align}
By property~\eqref{2.sepa}, we have $|\BB^{1/2}(u-v)|^2=|\BB^{1/2}(\widehat{u}-v)|^2=\la\mu,|\BB^{1/2}(s-v)|^2\ra$.
Thus, the coercivity estimate from Lemma \ref{lem.coerc} in Appendix \ref{sec.app} implies that
$$
  \int_\Omega\langle\mu_{x,t},|s-v(x,t)|^2\rangle\dd x \le CH_{\rm rel}^{\rm mv}(u(t)|v(t)).
$$
We insert this bound into~\eqref{eq:ediss-rel}
and invoke Gronwall's inequality to deduce that
\begin{align*}
  H_{\rm rel}^{\rm mv}(u(t)|v(t)) &+\int_0^t\int_\Omega\bigg\langle\mu_{x,\tau},
  \frac12|\BB^{1/2}(p-\na \widehat{v}(x,\tau))|^2 
  \bigg\rangle\dd x\dd \tau \le {\rm e}^{Ct}H_{\rm rel}(u^{\rm in}|v(0)) =0,
\end{align*}
where the last equality follows from $v(0)=u^{\rm in}$.
Hence, $\mu_{x,t}=\delta_{v(x,t)}\otimes
\delta_{\na\widehat{v}(x,t)}$ for a.e.\ $(x,t)\in\Om\times(0,T)$, which finishes the proof of Theorem \ref{thm.wsu}.

\section{Long-time asymptotics}
\label{sec.time}

In this section, we prove Theorem~\ref{thm.time}.
First, we verify that $\mathfrak{S}_\mathsf{m}\subset L^\infty(\Omega)$. Indeed, if $v\in \mathfrak{S}_\mathsf{m}$,
the vector $\BB v$ is constant and $\int_\Omega \BB v\dd x = \BB \mathsf{m}$,
which implies that $\BB v=(\BB \mathsf{m})/|\Omega|$.
Since the entries of $\BB$ and the components of $v$ are nonnegative,
$v_i \le (\BB \mathsf{m})_i/(b_{ii}|\Omega|)$ for all $i=1,\ldots,n$. This proves the 
claim.

The entropy inequalities \eqref{1.eiBm}--\eqref{1.eiRm} and the bound
$|\la\mu,B^{1/2}p\ra|^2\le \la\mu,|B^{1/2}p|^2\ra$, which follows from Jensen's inequality, show that
$$
  \int_0^\infty\|\na(\BB^{1/2}u)\|_{L^2(\Omega)}^2\dd t < \infty, \quad
	\sup_{0<t<\infty}\|u(t)\|_{L^2(\Omega)} < \infty.
$$
Thus, there exists a sequence $(t_k)\subset(0,\infty)$ with $t_k\to\infty$
such that $u(t_k)\rightharpoonup u^*$ weakly in $L^2(\Omega)$ and
$\BB^{1/2}u(t_k)\to \BB^{1/2}u^*$ strongly in $L^2(\Omega)$ as $k\to\infty$.
Since $\int_\Omega u(t_k)\dd x=\mathsf{m}$ and the sequence
$(\na(\BB^{1/2}u(t_k)))$ converges to zero 
in the $L^2(\Omega)$ norm, we find that 
$\int_\Omega u^*\dd x=\mathsf{m}$ and $\na(\BB^{1/2}u^*)=0$.
This implies that $u^*\in \mathfrak{S}_\mathsf{m}$. 
Moreover, we deduce from the strong convergence that
$$
  \lim_{k\to\infty}H_R(u(t_k)|u^*) = \frac12\lim_{k\to\infty}
	\|\BB^{1/2}(u(t_k)-u^*)\|_{L^2(\Omega)}^2 = 0.
$$

We assert that $t\mapsto H_R(u(t)|u^*)$ is nonincreasing for a.e.\ $t>0$. Indeed, we know from  Section \ref{sec.ent.ineq}
that $t\mapsto H_R(u(t))$ is nonincreasing. Furthermore, since
$\int_\Omega u(t)\dd x=\int_\Omega u^*\dd x$ and $\BB u^*$ is a constant vector,
we have $\int_\Omega u(t)^T\BB u^*\dd x = \int_\Omega u(s)^T \BB u^*\dd x$ for all $s,t\ge0$. Hence, for $t\ge s$,
\begin{align*}
  H_R(u(t)|u^*) &= H_R(u(t)) + H_R(u^*) - \int_\Omega u(t)^T\BB u^*\dd x \\
	&\le H_R(u(s)) + H_R(u^*) - \int_\Omega u(s)^T \BB u^*\dd x = H_R(u(s)|u^*),
\end{align*}
proving the claim.

We conclude that $H_R(u(t)|u^*)\le H_R(u(t_k)|u^*)\to 0$ for $t\ge t_k\to\infty$.
It follows from the positive definiteness of $\BB^{1/2}$ on $L^\perp$ that
$$
  \|\widehat{u}(t)-\widehat{u}^*\|_{L^2(\Omega)} 
	\le C\|\BB^{1/2}(\widehat{u}(t)-\widehat{u}^*)\|_{L^2(\Omega)}
	\le 2H_R(u(t)|u^*)\to 0
$$
as $t\to\infty$. This finishes the proof of Theorem \ref{thm.time}.

\begin{appendix}
\section{Auxiliary results}\label{sec.app}

Let the matrix $\BB=( b_{ij})\in\R^{n\times n}$ be symmetric 
positive semidefinite. Then the square root
of $\BB$ exists and $z^T\BB z=|\BB^{1/2}z|^2$ for $z\in\R^n$.
Let $P_L$ and $P_{L^\perp}$ be the projection matrices onto 
$L=\operatorname{ker}\BB=\operatorname{ker}\BB^{1/2}\neq\{0\}$ and 
$L^\perp=\operatorname{ran}\BB$, respectively.

\begin{lemma}\label{lem.APL}
Let $\lambda>0$ be the smallest positive eigenvalue of $\BB^{1/2}$. Then
$$
  |P_{L^\perp}z| \le \lambda^{-1}|\BB^{1/2}z| \quad\mbox{for }z\in\R^n.
$$
\end{lemma}

\begin{proof}
Let $z\in\R^n$ and $\widehat{z}=P_{L^\perp}z$. 
By definition of $\lambda$, $|\BB^{1/2}\widehat{z}|\ge\lambda|\widehat{z}|$.
Then the conclusion follows from
$\BB^{1/2}\widehat{z}=\BB^{1/2}z - \BB^{1/2}P_Lz = \BB^{1/2}z$.
\end{proof}

We introduce the relative entropy densities
\begin{align*}
  h_\Shn(u|v) &= \sum_{i=1}^n \big(\shn(u_i) - \shn(v_i) - \shn'(v_i)(u_i-v_i) 
	= \sum_{i=1}^n \bigg(u_i\log\frac{u_i}{v_i} - (u_i-v_i)\bigg), \\
	h_R(u|v) &= \frac12(u-v)^T\BB(u-v) = \frac12|\BB^{1/2}(u-v)|^2,
	\quad u,v\in[0,\infty)^n,
\end{align*}
where $\shn(z)=z(\log z-1)+1$. We denote by $\|A\|_2$ the norm of $A$ induced by
the Euclidean norm $|\cdot|$ in $\R^n$.

\begin{lemma}[Coercivity]\label{lem.coerc}
Let $a_0=\frac12\min_{i\in\{1,\ldots,n\}} b_{ii}>0$, 
$a_1=\|\BB\|_2$, and let $K\ge 1$. Then there exists a constant $c_*>0$, only depending on $a_0$, $a_1/a_0$,
and $M$, such that for all $u$, $v\in\Rge^n$ with $0<|v|\le M$,
$$
  h_\Shn(u|v) + h_R(u|v) \ge c_*|u-v|^2.
$$
\end{lemma}

\begin{proof}
By assumption, we have 
$\frac12 u^T\BB u\ge\frac12\sum_{i=1}^n b_{ii}u_i^2\ge a_0|u|^2$ for all 
$u\in\Rge^n$. If $(a_0/2)|u|\ge a_1|v|$ then
\begin{align*}
  h_R(u|v) &= \frac12 u^T\BB v - v^T\BB u + \frac12 v^T\BB v
	\ge a_0|u|^2 - a_1|u||v| + a_0|v|^2 \\
	&\ge a_0|u|^2 - \frac{a_0}{2}|u|^2 + a_0|v|^2 = \frac{a_0}{2}|u|^2 + a_0|v|^2
	\ge \frac{a_0}{3}|u-v|^2.
\end{align*}
Next let $(a_0/2)|u|<a_1|v|$. We find for $f(z)=z\log z$ that
\begin{align*}
  u_i\log&\frac{u_i}{v_i}- (u_i-v_i) 
	= f(u_i) - f(v_i) - f'(v_i)(u_i-v_i) \\
	&= (u_i-v_i)\int_0^1 
	\big(f'(s(u_i-v_i)+v_i)-f'(v_i)\big)\big|_{s=0}^\theta\dd\theta \\
	&= (u_i-v_i)^2\int_0^1\int_0^\theta f''(s(u_i-v_i)+v_i)\dd s\dd\theta.
\end{align*}
Then we infer from $|u_i/v_i|<2a_1/a_0$ that
$$
  f''(s(u_i-v_i)+v_i) = \frac{1}{v_i(s(u_i/v_i-1)+1)} 
	> \frac{1}{M(s(2a_1/a_0-1)+1)}
$$
and consequently,
$$
  u_i\log\frac{u_i}{v_i} - (u_i-v_i) 
	\ge \frac{(u_i-v_i)^2}{M}\int_0^1\int_0^\theta\frac{\dd s\dd\theta}{s(2a_1/a_0-1)+1},
$$
which shows that $h_\Shn(u|v) \ge c_1|u-v|^2$, where
$$
  c_1 = \frac{1}{M}\min_{i=1,\ldots,n} \int_0^1\int_0^\theta
	\frac{\dd s\dd\theta}{s(2a_1/a_0-1)+1}.
$$
Putting these estimates together and observing that $h_\Shn(u|v)\ge 0$, $h_R(u|v)\ge 0$,
we conclude the proof with $c_*=\min\{a_0/3,c_1\}$. 
\end{proof}

\section{A discrete Aubin--Lions compactness result}\label{app:compactness}
We here summarize the compactness result~\cite[Theorem~3.4]{GaLa12}, which is a basic ingredient in the proof of the convergence of our numerical approximation scheme (cf.\ Section~\ref{sec.conv}). 
We focus on the specific functional setting that is needed for our purpose.
In this setting, the proof of~\cite[Theorem~3.4]{GaLa12} relies on the following two key properties, whose validity has been verified in the proof of Proposition~9 of~\cite[Section 6.1]{JuZu20}:
\begin{enumerate}
	\item[(P1)] Let $(v_m)_m$ be a sequence of functions with $v_m\in V_{\T_m}$ for all $m$ and such that $\sup_m\|v_m\|_{1,2,\T_m}<\infty$. Then there exists a function $v\in L^2(\Om)$ such that, along a subsequence, $v_m\to v$ in $L^2(\Om)$.
	\item[(P2)] If $v_m\to v$ in $L^2(\Om)$ and $\|v_m\|_{-1,4,\T_m}\to0$, then $v\equiv0$.
\end{enumerate}
Recall that, for a given spatial mesh $\T$ and a time step size $\Delta t$,  the discrete function spaces $V_\T$ and $V_{\T,\Delta t}$ were defined in Section~\ref{ssec:function-spaces}.
Thanks to (P1) and (P2), the specific version of Theorem 3.4 in~\cite{GaLa12}, which we rely upon, can be stated as follows.
\begin{proposition}[Corollary of Theorem 3.4 in \cite{GaLa12}]\label{prop:comp.app}
Let $(\widehat{u}_m)_m$ be a sequence of functions such that $\widehat{u}_m\in  V_{\T_m,\Delta t_m}$ for all  $m\in \mathbb{N}$. 
Suppose that there exists a finite constant $C>0$ such that for all $m\in \mathbb{N}$,
\[
\sum_{k=1}^{N_m}\Delta t_m\|\widehat{u}^k_m\|_{1,2,\T_m}^2 
+\sum_{k=1}^{N_m}\Delta t_m\|\pa_t^{\Delta t_m}\widehat{u}^k_m\|_{-1,4,\T_m}^2
\le C.
\]
Then there exists $v\in L^2(0,T;L^2(\Om))$ such that, after passing to a subsequence, 
\[\widehat{u}_m\to v\text{  in }L^2(0,T;L^2(\Om))\;\;\text{as }m\to\infty.\]
\end{proposition}
The key point which this result addresses, as compared to more classical versions of the Aubin--Lions lemma, is its ability to handle a dependency of the spatial norms on the parameter $m$ itself.
\end{appendix}

\end{document}